\documentclass[12pt,twoside]{amsart}
\usepackage{amssymb,amsmath,amsfonts,amsthm}
\usepackage{mathrsfs}
\usepackage[X2,T1]{fontenc}
\usepackage[applemac]{inputenc}
\usepackage{cite}
\usepackage{latexsym}
\usepackage{verbatim}
\usepackage{soul}
\usepackage[dvipsnames]{xcolor}
\def\Im{\mathop{\rm Im}\nolimits}
\def\Re{\mathop{\rm Re}\nolimits}

\def\Im{\mathop{\rm Im}\nolimits}
\def\Re{\mathop{\rm Re}\nolimits}

\def\R{\mathbb R}

\def\N{\mathbb N}

\def\pxi{\langle \xi \rangle}
\def\px{\langle x \rangle}
\def\ds{\displaystyle}

\newcommand\dslash{d\llap {\raisebox{.9ex}{$\scriptstyle-\!$}}}
\usepackage{colortbl}

\newcommand{\beqsn}{\arraycolsep1.5pt\begin{eqnarray*}}
\newcommand{\eeqsn}{\end{eqnarray*}\arraycolsep5pt}
\newcommand{\beqs}{\arraycolsep1.5pt\begin{eqnarray}}
\newcommand{\eeqs}{\end{eqnarray}\arraycolsep5pt}

\newtheorem{theorem}{Theorem}
\newtheorem{lemma}{Lemma}

\newtheorem{proposition}{Proposition}
\newtheorem{definition}{Definition}

\newtheorem{remark}{Remark}

\catcode`\@=11

\renewcommand{\section}%
   {\setcounter{equation}{0}\@startsection {section}{1}{\z@}{-3.5ex plus -1ex
  minus -.2ex}{2.3ex plus .2ex}{\Large\bf}}

\title[Gevrey well-posedness for $3$-evolution equations ]{Gevrey well-posedness for $3$-evolution equations with variable coefficients}

\author[A. Arias Junior]{Alexandre Arias Junior}
\address{ Department of Mathematics, Federal University of Paran\'a, CEP 81531-980, Curitiba, Brazil}
\email{arias@ufpr.br}

\author[A. Ascanelli]{Alessia Ascanelli}
\address{Dipartimento di Matematica ed Informatica\\Universit\`a di Ferrara\\
Via Machiavelli 30\\
44121 Ferrara\\
Italy}
\email{alessia.ascanelli@unife.it}

\author[M. Cappiello]{Marco Cappiello}
\address{Dipartimento di Matematica ``G. Peano'' \\Universit\`a di Torino\\
Via Carlo Alberto 10\\
10123 Torino\\
Italy}
\email{marco.cappiello@unito.it}

\thanks{The first author was financed in part by the Coordena\c{c}\~ao de Aperfei\c{c}oamento de Pessoal de N\'ivel Superior - Brasil (CAPES) - Finance Code 001. The second author has been supported by the italian national research funds FFABR2017 and GNAMPA2020. The third author has been supported by the fund GNAMPA2020.}

\begin{document}


\begin{abstract}
We study the Cauchy problem for a class of third order linear anisotropic evolution equations with complex-valued lower order terms depending both on time and space variables. Under suitable decay assumptions for $|x| \to \infty$ on these coefficients, we prove a well-posedness result in Gevrey-type spaces.
\end{abstract}

\maketitle



\section{Introduction and main result}\label{section_introduction}
In this paper we deal with well-posedness of the Cauchy problem in Gevrey-type spaces for a class of linear anisotropic evolution operators of the form
\begin{equation}\label{diffP}
P(t, x, D_t, D_x) = D_t  + a_3(t, D_x) + a_2(t,x,D_x)+a_1(t,x,D_x)+ a_0(t,x,D_x),
\end{equation}
with  $(t,x) \in [0,T] \times \R$,  where $a_j(t,x,D_x)$ are pseudodifferential operators of order $j, j=0,1,2,3$ with complex-valued symbols and the symbol $a_3$ is independent of $x$ and real-valued. The latter condition guarantees that the operator $P$ satisfies the assumptions of the Lax-Mizohata theorem (see \cite[Theorem 3 page 31]{mizohata2014}), since the principal symbol of $P$  has the real characteristic $\tau=-a_3(t,\xi)$.
Operators of the form \eqref{diffP} can be referred to as $3$-evolution operators since they are included in the general class of $p$-evolution operators
$$P(t, x, D_t, D_x) = D_t  + a_p(t,D)+\sum_{j=0}^{p-1}a_j(t,x, D_x)$$
 introduced by Mizohata, cf. \cite{Mizo}, where $a_j$ are operators of order $j, j=0, \ldots, p,$  $p$ being a positive integer, and $a_p(t,\xi)$ is real-valued. Sufficient conditions for the well-posedness of the Cauchy problem 
in $H^\infty = \cap_{m \in \R}H^m$ for $p$-evolution operators have been proved in \cite{ascanelli_chiara_zanghirati_2012, cicognani_colombini_cauchy_problem_for_p_evolution_equations} for arbitrary $p$, whereas in the realm of Gevrey classes the results are limited to the case $p=2$, corresponding to Schr\"odinger operators, cf. \cite{ACR,CRJEECT, KB}.

In this paper we consider the case $p =3$ in the Gevrey setting. Third order linear evolution equations have a particular interest in mathematics since they can be regarded as linearizations of relevant physical semilinear models like KdV and KdV-Burgers equation and their generalizations, see for instance \cite{JM, Kato, KdV, KKN,TMI} . There are some results  concerning KdV-type equations  with coefficients not depending on $(t,x)$ in the Gevrey setting, see \cite{GHGP, Goubet, HHP}.  Our aim is to propose a general approach for the study of the Gevrey well-posedness for general linear and semilinear $3$-evolution equations with variable coefficients. Moreover, we want to consider the case when the coefficients of the lower order terms are complex-valued. In this situation it is well known that suitable assumptions for $|x| \to \infty$ on the coefficients of the lower order terms are needed to obtain well-posedness results, cf. \cite{ascanelli_chiara_zanghirati_necessary_condition_for_H_infty_well_posedness_of_p_evo_equations, ichinose_remarks_cauchy_problem_schrodinger_necessary_condition}. In this paper we establish the linear theory for the Cauchy problem associated to an operator of the form \eqref{diffP} whereas the semilinear case is treated in \cite{AACJMPA} where we adapt to the Gevrey framework the method proposed in 
\cite{DA} for hyperbolic equations and in \cite{ABtame} for $p$-evolution equations in the $H^\infty$ setting. There, well-posedness for the semilinear Cauchy problem is studied first by considering a linearized problem and then by deriving well-posedness for the semilinear equation using Nash-Moser inversion theorem. Hence, the results obtained in this paper, apart their interest per se for the linear theory, are also a preliminary step for the study of the semilinear case.

Let us consider for $(t,x) \in [0,T] \times \R$ the Cauchy problem in the unknown $u=u(t,x)$:
\begin{equation}\label{Cauchy_problem_in_introduction}
\begin{cases}
P(t, x, D_t, D_x) u(t,x)= f(t, x),	\\
u(0,x) = g(x), 
\end{cases}
\end{equation}
with $P$ defined by \eqref{diffP}. 

Concerning the functional setting, fixed $\theta \geq 1, m, \rho \in \R,$  we set
$$
H^{m}_{\rho; \theta} (\R) = \{ u \in \mathscr{S}'(\R) :  \langle D \rangle^{m} 
e^{\rho \langle D \rangle^{\frac{1}{\theta}}} u \in L^{2}(\R) \},
$$
where $\langle D \rangle^m$ and $e^{\rho \langle D \rangle^{\frac{1}{\theta}}}$ are the Fourier multipliers with symbols $\langle \xi \rangle^m$ and $e^{\rho \langle \xi \rangle^{\frac{1}{\theta}}}$ respectively.
These spaces are Hilbert spaces with the following inner product
$$\langle u, v \rangle_{H^{m}_{\rho;\theta}} = \, \langle\langle D \rangle^{m} e^{\rho\langle D \rangle^{\frac{1}{\theta}}}u, \langle D \rangle^{m} e^{\rho\langle D \rangle^{\frac{1}{\theta}}}v \rangle_{L^{2}}, \quad u, v \in H^{m}_{\rho;\theta}(\R).
$$
We want to investigate the well-posedness of the problem \eqref{Cauchy_problem_in_introduction} in the space
$$ \mathcal{H}^\infty_{\theta} (\R): = \bigcup_{\rho >0}H^m_{\rho; \theta}(\R).$$ 
This space is related to Gevrey classes in the following sense: it is easy to verify the inclusions
$$ G_0^\theta(\R) \subset  \mathcal{H}^\infty_{\theta} (\R) \subset G^\theta(\R),$$
where $G^\theta(\R)$ 
denotes the space of all smooth functions $f$ on $\R$ such that 
\begin{equation}
\label{gevestimate}
\sup_{\alpha \in \N} \sup_{x \in \R} h^{-|\alpha|} \alpha!^{-\theta} |\partial^\alpha f(x)| < +\infty
\end{equation} 
for some $h >0$, and $G_0^\theta(\R)$ is the space of all compactly supported functions contained in $G^\theta(\R)$.\\
The terms $a_j(t,x,D_x)$ are assumed to be pseudodifferential operators with Gevrey regular symbols; of course differential operators $a_j(t,x)D_x^j$, with $a_j(t,x)$ continuous on $t$ and Gevrey regular on $x$, are a particular case. Let us introduce a suitable symbol class.

\begin{definition} \label{standardsymbols}
For fixed $m \in \R, \mu \geq 1$, $\nu \geq 1$ and $A>0$, we shall denote by $S^m_{\mu,\nu}(\R^{2n};A)$ (respectively $\tilde{S}^m_{\mu,\nu}(\R^{2n};A)$) the Banach space of all functions $a \in C^\infty(\R^{2n})$ satisfying the following estimate
$$
\|a\|_A:= \sup_{\alpha, \beta \in \N_{0}^n} \sup_{(x,\xi) \in \R^{2}} A^{-|\alpha|-|\beta|} \alpha!^{-\mu} \beta!^{-\nu} \pxi^{-m+|\alpha|}|\partial_\xi^\alpha \partial_x^\beta a(x,\xi) | <+\infty,
$$ 
(respectively
$$
|a |_A:= \sup_{\alpha, \beta \in \N_{0}^n} \sup_{(x,\xi) \in \R^{2n}} A^{-|\alpha|-|\beta|} \alpha!^{-\mu} \beta!^{-\nu} \pxi^{-m}|\partial_\xi^\alpha \partial_x^\beta a(x,\xi) | <+\infty).
$$ 

We set $S^{m}_{\mu,\nu}(\R^{2n}) = \displaystyle\bigcup_{A > 0}S^m_{\mu,\nu}(\R^{2n};A)$ and $\tilde{S}^{m}_{\mu,\nu}(\R^{2n}) = \displaystyle\bigcup_{A > 0}\tilde{S}^m_{\mu,\nu}(\R^{2n};A)$ endowed with the inductive limit topology. In the case $\mu=\nu$ we simply write $S^m_{\mu}(\R^{2n})$ and $\tilde{S}^m_{\mu}(\R^{2n})$  instead of $S^m_{\mu,\mu}(\R^{2n})$ and $\tilde{S}^m_{\mu,\mu}(\R^{2n})$.
\end{definition}

The main result of the paper reads as follows.

\begin{theorem}\label{mainthm}
	Let $s_0 > 1$ and $\sigma \in (\frac{1}{2}, 1)$ such that $s_0 < \frac{1}{2(1-\sigma)}$. Let moreover $P(t,x,D_t, D_x)$ be defined by \eqref{diffP}. Assume that 
\begin{itemize}
	\item [(i)] $a_3(t,\xi) \in C([0,T]; S^{3}_{1}(\R^2))$, $a_3(t,\xi)$ is real-valued, 
	and there exists $R_{a_3} > 0$ such that
	$$
	|\partial_{\xi} a_3(t,\xi)| \geq C_{a_3} \xi^{2}, \quad t \in [0,T],\,\, |\xi| > R_{a_3};
	$$
	
	\item [(ii)] $a_j \in C([0,T]; S^{j}_{1,s_0}(\R^{2}))$ for $j = 0,1,2$; 
	
	\item [(iii)] there exists $C_{a_2} > 0$ such that  
	$$
	|\partial^{\alpha}_{\xi} \partial^{\beta}_{x} a_2(t,x,\xi)| \leq C^{\alpha+\beta+1}_{a_2} \alpha! \beta!^{s_0} \langle \xi \rangle^{2-\alpha} \langle x \rangle^{-\sigma}, \quad t \in [0,T], \, x, \xi \in \R, \, \alpha, \beta \in \N_{0};
	$$
	
	\item [(iv)] there exists $C_{a_1}$ such that 
	$$
	|\Im\, a_1(t,x,\xi)| \leq C_{a_1} \langle \xi \rangle \langle x \rangle^{-\frac{\sigma}{2}}, \quad t \in [0,T], \, x,\xi \in \R.
	$$
\end{itemize}
Then given $\theta\in\left[s_0,\frac1{2(1-\sigma)}\right), \rho >0$, $m\in\R$, and given $f \in C([0,T], H^m_{\rho; \theta}(\R))$ and $g \in  H^m_{\rho; \theta}(\R)$, there exists a unique solution $u \in C^1([0,T], H^{m}_{\rho'; \theta}(\R))$ of \eqref{Cauchy_problem_in_introduction} for some $0< \rho' < \rho$, and $u$ satisfies the energy estimate
\begin{equation}
\label{energyestimate}
\| u(t, \cdot )\|_{H^{m}_{\rho'; \theta}}^2 \leq C \left( \| g \|^2_{H^{m}_{\rho; \theta}} + \int_0^t \| f (\tau, \cdot) \|^2_{H^{m}_{\rho; \theta}}\, d\tau \right).
\end{equation}
Moreover,
for $\theta \in [s_0, \frac{1}{2(1-\sigma)})$ the Cauchy problem \eqref{Cauchy_problem_in_introduction} is well-posed in $\mathcal{H}^\infty_{\theta}(\R) $.	
\end{theorem}

\begin{remark}
We notice that the solution $u$ exhibits a loss of regularity with respect to the initial data in the sense that it belongs to $H^m_{\rho';\theta}$ for some $\rho'<\rho.$ Moreover, the decay rate $\sigma$ of the coefficients imposes restrictions on the values of $\theta$ for which the Cauchy problem \eqref{Cauchy_problem_in_introduction} is well-posed. Such phenomena are typical of this type of problems and they appear also in the papers \cite{CRJEECT, KB}. In the recent paper \cite{AAC3evolGelfand-Shilov} we proved the existence of a solution with no loss of regularity with respect to the initial data and no upper bound for $\theta$ provided that the Cauchy data also satisfy a suitable exponential decay condition.
\end{remark}

\begin{remark} 
Let us make some comments on the decay assumptions (iii) and (iv) in Theorem \ref{mainthm}. For $p=2$ we know from \cite{KB} that the decay condition $\Im a_{p-1}(t,x)\sim\langle x\rangle^{-\sigma}$, $\sigma\in (0,1)$, leads to Gevrey well-posedness for $s_0\leq \theta<1/(1-\sigma)$. Here we prove that for $p=3$ the decay condition $a_{p-1}(t,x)\sim\langle x\rangle^{-\sigma}$, $\sigma\in (1/2,1)$, together with a weaker decay assumption on the (lower order) term $\Im a_{p-2}(t,x)\sim\langle x\rangle^{-\sigma/2}$ is sufficient to obtain Gevrey well-posedness for $s_0\leq \theta<1/(2(1-\sigma))$. Comparing these results, we point out that  in the present paper we need to assume that also $\Re a_2 \sim \langle x\rangle^{-\sigma}$ in order to control the term appearing in \eqref{equation_remainde_conj_level_2_by_k_D}. This assumption is crucial in the argument to obtain our result. In the recent paper \cite{AACJDE} we proved that for $\sigma \in (0, 1/2]$ the Cauchy problem \eqref{Cauchy_problem_in_introduction} is not well-posed in $\mathcal{H}^\infty_\theta(\R)$ for any $\theta$.
\end{remark}

To prove Theorem \ref{mainthm} we need to perform a suitable change of variable. In fact, if we set
$$
iP = \partial_{t} + ia_3(t,D) + \underbrace{\sum_{j = 0}^{2} ia_j(t,x,D_x)}_{=: A},
$$
since $a_3(t,\xi)$ is real-valued, we have 
\begin{align*}
	\frac{d}{dt} \| u(t) \|^{2}_{L^{2}} &= 2Re\, \langle \partial_{t} u(t), u(t)\rangle_{L^{2}} \\ 
	&= 2Re\, \langle iPu(t), u(t)\rangle_{L^2} - 2Re\, \langle ia_3(t,D) u(t), u(t) \rangle_{L^2} - 2Re\,\langle Au(t), u(t) \rangle_{L^2} \\
	&\leq \| Pu(t) \|^{2}_{L^{2}} + \| u(t) \|^{2}_{L^{2}} - \,\langle (A+A^{*})u(t), u(t)\rangle_{L^2}.
\end{align*}
However, $A+A^{*}$ is an operator of order $2$, so we cannot derive an energy inequality in $L^2$ from the estimate above. The idea is then to conjugate the operator $iP$ by a suitable invertible pseudodifferential operator $e^{\Lambda}(t,x,D_x)$ in order to obtain 
$$
(iP)_{\Lambda} := e^{\Lambda} \circ (iP) \circ \{e^{\Lambda}\}^{-1} = \partial_{t} + ia_3(t,D_x) + \{a_{2,\Lambda} + a_{1,\Lambda} + a_{\frac{1}{\theta}, \Lambda} + r_{0, \Lambda} \}(t,x,D_x),
$$
with $a_{j, \Lambda}(t,x,\xi)$ of order $j$ but with $Re\, a_{j, \Lambda} \geq 0$, for $j = \frac{1}{\theta}, 1, 2 $, and $r_{0,\Lambda}(t,x,D)$ has symbol $r_{0,\Lambda}(t,x,\xi)$ of order zero. In this way, applying Fefferman-Phong inequality to $a_{2,\Lambda}$ (see \cite{fefferman_Phong_inequality}) and sharp G{\aa}rding inequality to $a_{1,\Lambda}$ and $a_{\frac{1}{\theta}, \Lambda}$ (see Theorem $1.7.15$ of \cite{nicola_rodino_global_pseudo_diffferential_calculus_on_euclidean_spaces}),  we obtain the estimate from below  
$$
Re\, \langle (a_{2,\Lambda} + a_{1,\Lambda} + a_{\frac{1}{\theta}, \Lambda})(t,x,D_x) v(t), v(t) \rangle_{L^{2}} \geq -c \| v(t) \|^{2}_{L^{2}},
$$ 
and therefore for the solution $v$ of the Cauchy problem associated to the operator $P_\Lambda$ we get
$$
\frac{d}{dt} \| v(t) \|^{2}_{L^{2}} \leq 
C(\| (iP)_{\Lambda}v(t) \|^{2}_{L^{2}} + \| v(t) \|^{2}_{L^{2}}).
$$
Gronwall inequality then gives  the desired energy estimate for the conjugated operator $(iP)_{\Lambda}$. By standard arguments in the energy method we then obtain that the Cauchy problem associated with $P_{\Lambda}$ is well-posed in any Sobolev space $H^{m}(\R)$.

Finally we turn back to our original Cauchy problem \eqref{Cauchy_problem_in_introduction}. The problem \eqref{Cauchy_problem_in_introduction} is in fact equivalent to the auxiliary Cauchy problem
\begin{equation}\label{equation_conjugated_cauchy_problem}
\begin{cases}
P_{\Lambda}(t, x, D_t, D_x) v(t,x) = e^{\Lambda}(t,x,D_x) f(t,x), \quad (t,x) \in [0,T] \times \R, \\
v(0,x)= e^{\Lambda}(0,x, D_x)g(x), \quad x \in \R,
\end{cases}
\end{equation}
in the sense that if $u$ solves \eqref{Cauchy_problem_in_introduction} then $v=e^{\Lambda}(t,x,D_x)u$ solves (\ref{equation_conjugated_cauchy_problem}), and if $v$ solves (\ref{equation_conjugated_cauchy_problem}) then $u=\{e^{\Lambda}(t,x,D_x)\}^{-1}v$ solves \eqref{Cauchy_problem_in_introduction}. In this step the continuous mapping properties of $e^{\Lambda}(t,x,D_x)$ and $\{e^{\Lambda}(t,x,D_x)\}^{-1}$ play an important role.

The operator $e^\Lambda(t,x,D_x)$ will be the composition of two pseudodifferential operators of infinite order, namely
\begin{equation} \label{nuovadefLambda}
e^\Lambda(t, x, D_x) = e^{k(t) \langle D_x \rangle^{\frac{1}{\theta}}_{h}} \circ e^{ \tilde{\Lambda}}(x,D_x)
\end{equation}
where $\tilde{\Lambda} = \lambda_2 + \lambda_1 \in S^{2(1-\sigma)}_{\mu} (\R^2)$, $k \in C^1([0,T]; \R)$ is a positive non increasing function to be chosen later on and $\langle \xi \rangle_{h} := \sqrt{h^{2} + \xi^{2}}$ with $h > 0$ a large parameter. Now we briefly explain the main role of each part of the change of variables. The transformation with $\lambda_2$ will change the terms of order $2$ into the sum of a positive operator of order $2$ plus a remainder of order $1$; the transformation with $\lambda_1$ will not change the terms of order $2$, but it will turn the terms of order $1$ into the sum of a positive operator of the same order plus a remainder of order not exceeding $1/\theta$. Finally the transformation with $k$ will correct this remainder term. We also observe that since $2(1-\sigma) < 1/\theta$ the leading part is $k(t) \langle \xi \rangle^{\frac{1}{\theta}}_{h}$, hence the inverse of $e^{\Lambda}(t,x,D_x)$ possesses regularizing properties with respect to the spaces $H^m_{\rho; \theta}$, because $k(t)$ has positive sign.

The paper is organized as follows. In Section \ref{pseudos} we introduce a class of pseudodifferential operators of infinite order which includes the operator $e^{\Lambda}(t,x,D)$ mentioned above and its inverse and state a conjugation theorem. To introduce quickly the reader to the core of the paper, we report in this section only the main definitions and results and postpone long proofs and collateral results to the Appendix at the end of the paper.
In Section \ref{changeofvariables} we introduce the operator $e^{\tilde{\Lambda}}(x,D)$  and prove its invertibility. In Section  \ref{section_conjugation_of_iP} we treat the conjugation of $iP$ with $e^{\Lambda}(t,x,D)$ and its inverse and transform the Cauchy problem \eqref{Cauchy_problem_in_introduction} into \eqref{equation_conjugated_cauchy_problem}. Finally in Section \ref{Proofmainresult} we prove Theorem \ref{mainthm}.

\section{Pseudodifferential operators of infinite order}
\label{pseudos} 
In this section we recall some basic properties of pseudodifferential operators with symbols as in Definition \ref{standardsymbols}. Moreover, we introduce the pseudodifferential operators of infinite order which will be used to define the change of variable mentioned in the Introduction and state a conjugation theorem. 
Our approach follows the same ideas used in \cite{KN} but with some minor modifications due to the fact that we need a more explicit Taylor expansion of the symbol of the conjugated operator $(iP)_\Lambda$, cf. Remark \ref{remark6} below. Wanting to direct the reader to the main results of the article as soon as possible, we prefer to dedicate an Appendix at the end of the paper to discuss these technical facts. Notice that from now on we shall use the simpler notation $D$ instead of $D_x$ when denoting pseudodifferential operators since time derivatives are not involved in the expression of $e^\Lambda$.

In addition to the symbols defined in Definition 1, fixed $\theta>1$, $1\leq \mu\leq \theta$ and $A,c>0$, we will also consider the following Banach spaces: 
$$
p(x,\xi) \in S^{\infty}_{\mu;\theta}(\R^{2n};A,c) \iff \|p\|_{A,c} := \sup_{\overset{\alpha, \beta \in \N_{0}^{n}}{x,\xi \in \R^{n}}} |\partial_\xi^\alpha \partial_x^\beta p(x,\xi)| A^{-|\alpha+\beta|} \alpha!^{-\mu} \beta!^{-\mu} \langle \xi \rangle^{|\alpha|} e^{-c|\xi|^{\frac{1}{\theta}}} < +\infty;
$$
$$
p(x,\xi) \in \tilde{S}^{\infty}_{\mu;\theta}(\R^{2n};A,c) \iff |p|_{A,c} := \sup_{\overset{\alpha, \beta \in \N_{0}^{n}}{x,\xi \in \R^{n}}} |\partial_\xi^\alpha \partial_x^\beta p(x,\xi)| A^{-|\alpha+\beta|} \alpha!^{-\mu} \beta!^{-\mu} e^{-c|\xi|^{\frac{1}{\theta}}} < +\infty.
$$
We set  
$$S^\infty_{\mu;\theta}(\R^{2n}):= \bigcup_{c,A>0}S^{\infty}_{\mu;\theta}(\R^{2n};A), \qquad \tilde{S}^\infty_{\mu;\theta}(\R^{2n}):= \bigcup_{c,A>0}\tilde{S}^{\infty}_{\mu;\theta}(\R^{2n};A),$$
endowed with the inductive limit topology.

Unlike Definition \ref{standardsymbols}, the classes above have been defined for simplicity's sake assuming the same Gevrey regularity in $x$ and $\xi$ ($\mu=\nu$). This choice is justified by the fact that taking $\mu\neq\nu$ would not improve significantly the results here below.
Since in several parts of the paper we shall need more precise estimates for the $x$-derivatives of symbols, it is important to introduce also Gevrey regular $\textrm{\textbf{SG}}$ symbols, cf. \cite{CappielloRodino}. Namely, given $m_1,m_2 \in \R$, $\mu\geq 1$ and $A>0$, we say that $ p \in \textrm{\textbf{SG}}^{m_1,m_2}_{\mu}(\R^{2n}; A)$ if and only if $$ |||p |||_A :=
	\sup_{\stackrel{\alpha, \beta \in \N^{n}_{0}}{(x,\xi) \in \R^{2n}} }| \partial_\xi^\alpha \partial_x^\beta p(x,\xi) |A^{-|\alpha+\beta|} \alpha!^{-\mu} \beta!^{-\mu} \pxi^{-m_1+|\alpha|} \px^{-m_2+|\beta|}  <+\infty.
	$$ 
	We set $\textrm{\textbf{SG}}^{m_1,m_2}_{\mu}(\R^{2n}) = \displaystyle\bigcup_{A > 0}\textrm{\textbf{SG}}^{m_1,m_2}_{\mu}(\R^{2n};A)$ endowed with the inductive limit topology. 

\begin{remark}
We have the inclusions $S^{m}_{\mu}(\R^{2n}) \subset S^{\infty}_{\mu;\theta} (\R^{2n})$ for every $m \in \R$ and $\theta > 1$ and $\textrm{\textbf{SG}}_{\mu}^{m_1,m_2}(\R^{2n}) \subset S^{m_1}_{\mu}(\R^{2n})$ if $m_2 \leq 0.$
\end{remark}

Denote by $\gamma^\theta(\R^n)$ the space of all functions $f \in C^\infty(\R^n)$ such that \begin{equation}
	\label{gevestimate}
	\sup_{\alpha \in \N^n} \sup_{x \in \R^n} h^{-|\alpha|} \alpha!^{-\theta} |\partial^\alpha f(x)| < +\infty
\end{equation} 
for every $h >0$, and by $\gamma_0^\theta(\R^n)$ the space of all compactly supported functions contained in $\gamma^\theta(\R^n)$.
For a given symbol $p \in \tilde{S}^{\infty}_{\mu;\theta}(\R^{2n})$ we denote by $p(x,D)$ or by $\textrm{op}(p)$ the pseudodifferential operator defined by 
\begin{equation} \label{pseudop}
p(x,D) u (x) = \int e^{i\xi x} p(x,\xi) \widehat{u}(\xi) \dslash\xi, \quad u \in \gamma_0^\theta(\R^{n}),
\end{equation}
where $ \dslash\xi = (2\pi)^{-n}d\xi.$
Arguing as in \cite[Theorem 3.2.3]{Rodino_linear_partial_differential_operators_in_gevrey_spaces} or \cite[Theorem 2.4]{Zanghirati} it is easy to verify that operators of the form \eqref{pseudop} with symbols from $\tilde{S}^\infty_{\mu;\theta}(\R^{2n})$, with $\mu <\theta$, map continuously $\gamma_0^\theta(\R^n)$ into $\gamma^\theta(\R^n)$. However, it is convenient to work in a functional setting which is invariant under the action of these operators.
This is represented by Gelfand-Shilov spaces of type $\mathscr{S}$, cf. \cite{GS2}.

\begin{definition}
Given $\theta \geq 1$ and $A > 0$ we say that a function $f \in C^\infty(\R^n)$ belongs to $\mathcal{S}_{\theta, A} (\R^{n})$ if there exists $C > 0$ such that 
$$
|x^{\beta} \partial^{\alpha}_{x}f(x)| \leq C A^{|\alpha|+|\beta|}\alpha!^{\theta} \beta!^{\theta}, 
$$
for every $\alpha, \beta \in \N^{n}_{0}$ and $x \in \R^{n}$.  We define 
$$
\mathcal{S}_{\theta}(\R^{n}) = \bigcup_{A>0} \mathcal{S}_{\theta, A} (\R^{n}), \quad 
\Sigma_{\theta} (\R^{n}) = \bigcap_{A > 0} \mathcal{S}_{\theta, A} (\R^{n}).
$$ 
\end{definition} 
The norm 
$$
\| f \|_{\theta, A} \,= \sup_{\overset{x \in \R^{n}}{\alpha, \beta \in \N^{n}_{0}}} |x^{\beta} \partial^{\alpha}_{x}f(x)|A^{-|\alpha|-|\beta|}\alpha!^{-\theta}\beta!^{-\theta}, \quad  f \in \mathcal{S}_{\theta, A} (\R^{n}),
$$
turns $\mathcal{S}_{\theta, A}(\R^{n})$ into a Banach space, which allows to equip $\mathcal{S}_{\theta}(\R^{n})$ (resp. $\Sigma_{\theta}(\R^{n})$) with the inductive (resp. projective) limit topology coming from the Banach spaces $\mathcal{S}_{\theta, A}(\R^{n})$.

\begin{remark}
We can also define, for $M, \varepsilon > 0$, the Banach space $\mathcal{S}_{\theta}(\R^{n};M,\varepsilon)$ of all functions $f \in C^\infty(\R^n)$ such that
$$
\|f\|_{M,\varepsilon}:= \sup_{\stackrel{x \in \R^{n}}{\alpha \in \N^{n}_{0}}}M^{-|\alpha|} \alpha!^{-\theta} e^{\varepsilon |x|^{\frac{1}{\theta}}}|\partial^{\alpha}_{x}f(x)| <\infty ,
$$
and we have (with equivalent topologies) 
$$
\mathcal{S}_{\theta}(\R^{n}) = \bigcup_{M,\varepsilon >0} \mathcal{S}_{\theta} (\R^{n};M,\varepsilon), \quad 
\Sigma_{\theta} (\R^{n}) = \bigcap_{M, \varepsilon > 0} \mathcal{S}_{\theta} (\R^{n};M,\varepsilon).
$$ 
\end{remark}

It is easy to see that the following inclusions are continuous (for every $\varepsilon > 0$)
$$ \gamma_0^\theta(\R^n) \subset
\Sigma_{\theta} (\R^{n}) \subset  \mathcal{S}_{\theta}(\R^{n}) \subset G^{\theta} (\R^{n}),
\quad
\mathcal{S}_{\theta}(\R^{n}) \subset \Sigma_{\theta+\varepsilon}(\R^{n}),
$$
for every $\theta \geq 1.$
We shall denote by $(\mathcal{S}_{\theta})' (\R^{n})$, $(\Sigma_{\theta})' (\R^{n})$
the respective dual spaces. Concerning the action of the Fourier transform we have the following isomorphisms
$$ \mathcal{F}: \mathcal{S}_{\theta}(\R^{n}) \to \mathcal{S}_{\theta}(\R^{n}), \quad
\mathcal{F}: \Sigma_{\theta} (\R^{n}) \to \Sigma_{\theta} (\R^{n}).
$$

\begin{proposition}
Let $p \in \tilde{S}^{m}_{\mu}(\R^{2n};A)$. Then for every $\theta \geq  \mu$ (resp. $\theta >  \mu$), the operator $p(x,D)$ maps continuously $\mathcal{S}_{\theta}(\R^{n})$ into $\mathcal{S}_{\theta}(\R^{n})$ (resp. $\Sigma_{\theta}(\R^{n})$ into $\Sigma_{\theta}(\R^{n})$), and it extends to a continuous map from $(\mathcal{S}_{\theta})'(\R^{n})$ (resp. $(\Sigma_\theta)'(\R^n)$) into itself. Moreover, there exists $\tilde{\delta} >0$ such that $p(x,D)$ maps continuously $H^{m+m'}_{\rho;\theta}(\R^n)$ into $H^{m'}_{\rho;\theta}(\R^n)$ for every $m' \in \R$ and for $|\rho| < \tilde{\delta} A^{-1/\theta}.$  
\end{proposition}

\begin{proof} The first assertion can be proved following readily the argument in the proof of \cite[Theorem 2.2]{CappielloRodino}. The second one is the content of \cite[Proposition 6.3]{KN}.
\end{proof}

By \cite[Proposition 6.4]{KN}, given $p \in S^{m}_{\mu}(\R^{2n};A)$ and $q \in S^{m'}_{\mu}(\R^{2n};A)$, the operator $p(x,D)q(x,D)$ is a pseudodifferential operator with symbol $s$ given for every $N \geq 1$ by
$$s(x,\xi)= \sum_{|\alpha| <N}(\alpha!)^{-1}\partial_\xi^\alpha p(x,\xi) D_x^\alpha q(x,\xi) +r_N(x,\xi),$$
where $r_N$ satisfies
$$|\partial_\xi^\alpha D_\xi^\beta r_N(x,\xi)| \leq C_{N,A}(C_N A)^{|\alpha+\beta|}(\alpha! \beta!)^\mu \langle \xi \rangle^{m+m'-N-|\alpha|}$$
for every $x,\xi \in \R^n, \alpha, \beta \in \N_0^n,$ with $C_N$ independent of $A, \alpha,\beta.$

We shall not develop a complete calculus for pseudodifferential operators of infinite order here since for our purposes we can limit to consider some particular examples of such operators, namely defined by a symbol of the form $e^{\lambda(x,\xi)}$ for some $\lambda \in S^{1/\kappa}_{\mu}(\R^{2n}), \kappa> 1.$ In fact, 
let $\lambda$ be a real-valued symbol satisfying 
the following condition:
\beqs\label{condition_Lambda}|\partial_\xi^\alpha \partial_x^\beta \lambda(x,\xi)| \leq \rho_0 A^{|\alpha+\beta|}(\alpha!\beta!)^\kappa \langle \xi \rangle^{\frac1{\kappa}-|\alpha|},
\eeqs
It is easy to verify that $e^{\pm \lambda} \in S^\infty_{\kappa;\kappa}(\R^{2n}).$  Let
$$e^{\pm \lambda}(x,D) u(x) =\textrm{op}(e^{\pm \lambda})u(x) = \int_{\R^n} e^{i\xi x \pm\lambda(x,\xi)}\hat{u}(\xi)\, \dslash \xi.$$ 
We also consider the so-called reverse operator $^{R}\{e^{\pm \lambda}(x,D)\}$, introduced in \cite[Proposition 2.13]{KW} as the transposed of $e^{\pm \lambda}(x,-D)$, see also \cite{KN}. Namely, $^{R}\{e^{\pm \lambda}(x,D)\}$ is defined as an oscillatory integral by
\begin{eqnarray*}^{R}\{e^{\pm \lambda}(x,D)\}u(x) &=& Os - \iint e^{i\xi (x-y) \pm \lambda(y,\xi)}u(y)\, dy \dslash \xi \\
	&=& \lim_{\varepsilon \to 0} \iint_{\R^{2n}} e^{i\xi (x-y) \pm \lambda(y,\xi)} \chi(\varepsilon y, \varepsilon \xi)u(y)\, dy \dslash \xi
	\end{eqnarray*}
for some $\chi \in \mathcal{S}_\kappa(\R^{2n})$ such that $\chi(0,0)=1.$

The following continuity result holds for the operators $e^\lambda(x,D)$ and $^{R}\{e^\lambda(x,D)\}$:

\begin{proposition}\label{contgev}
	Let $\lambda$ be a symbol as in \eqref{condition_Lambda}, $\rho,m\in\R$, $1<\theta\leq\kappa$. Then: \\
	i)
	If $\kappa > \theta$, the operators $e^{\lambda}(x,D)$ and $^{R}\{e^{\lambda}(x,D)\}$ map continuously $H^{m}_{\rho; \theta} (\R^n)$ into $H^{m}_{\rho-\delta; \theta}(\R^n)$ for every $\delta >0$; \\
	ii) If $\kappa =\theta $, there exists $\tilde\delta>0$ such that the map $e^{\lambda}(x,D):H^{m}_{\rho; \theta} (\R^n)\longrightarrow H^{m}_{\rho-\delta; \theta}(\R^n)$ is continuous for every $|\rho-\delta|<\tilde\delta A^{-1/\theta}$ and  $$\delta >C(\lambda):= \sup \{\lambda(x,\xi) /\langle \xi \rangle^{1/\theta}: (x,\xi) \in \R^{2n} \}.$$
Moreover, $^{R}\{e^\lambda (x,D)\}:H^{m}_{\rho; \theta} (\R^n)\longrightarrow H^{m}_{\rho-\delta; \theta}(\R^n)$ is continuous for every $|\delta|<\tilde\delta  A^{-1/\theta}$ and $\delta>C(\lambda).$	
\end{proposition}
For the proof, see \cite[Proposition 6.7]{KN}.

\begin{definition}
Let $r>1.$ We denote by $\mathcal{K}_r$ the space of all $p \in C^\infty(\R^{2n})$ satisfying an estimate of the form
\begin{equation}\label{regestimate}|\partial_\xi^\alpha \partial_x^\beta p(x,\xi)| \leq C^{|\alpha+\beta|+1}\alpha!^r \beta!^r e^{-c|\xi|^{1/r}}
	\end{equation}
for some positive constants $C,c$. 
\end{definition}

\begin{remark} 
	Operators with symbols in $\mathcal{K}_{r}$ possess regularizing properties in the sense that they extend to linear and continuous maps from $(\gamma_0^\theta)'(\R^n)$ into  $\gamma^\theta(\R^n)$ if $\theta >r$. This can be easily proved taking into account that
	the estimate \eqref{regestimate} implies that for every $\varepsilon>0$ there exists $C_\varepsilon>0$ such that
	$$|\partial^{\alpha}_{\xi} \partial^{\beta}_{x}p(x,\xi)| \leq C_\varepsilon C^{|\alpha+\beta|}\alpha!^{r} \beta!^{r} e^{-\varepsilon|\xi|^{\frac{1}{\theta}}}, \quad x,\xi \in \R^{n}, \alpha,\beta \in \N^{n}_{0}.$$
	Moreover  the Fourier transform of an element of $(\gamma_0^\theta)'(\R^n)$ is a function which can be bounded by $Ce^{c|\xi|^{1/\theta}}$ for some $C,c>0$. Hence we can prove the continuity using the same argument of the proof of \cite[Lemma 3.2.12]{Rodino_linear_partial_differential_operators_in_gevrey_spaces}. We shall often refer to these operators as $r$-regularizing operators in the sequel.
\end{remark}

In the next result we shall need to work with the weight function $\langle\xi\rangle_h= (h^2+|\xi|^2)^{1/2}$
where $h \geq 1$. We point out that we can replace $\langle \xi \rangle$ by $\langle \xi \rangle_{h}$ in all previous definitions and statements, and this replacement does not change the dependence of the constants, that is, all the previous constants are independent of $h$. Moreover, we also need the following stronger hypothesis on $\lambda(x,\xi):$
\begin{equation}\label{equation_stronger_hypothesis_on_Lambda}
	|\partial_\xi^\alpha \partial_x^\beta \lambda(x,\xi)| \leq \rho_0 A^{|\alpha+\beta|}\alpha!^{\kappa}\beta!^{\kappa} \langle \xi \rangle^{-|\alpha|}_{h}, 
\end{equation}
whenever $|\beta| \geq 1$. This means that if at least an $x-$derivative falls on $\lambda$, then we obtain a symbol of order $0 (< \frac{1}{\kappa})$. 

\begin{theorem}\label{conjthmnew}
	Let $p $ be a symbol satisfying 
	$$
	|\partial_\xi^\alpha \partial_x^\beta p(x,\xi)| \leq C_A A^{|\alpha+\beta|} \alpha!^{\kappa}\beta!^{\kappa} \langle \xi \rangle^{m-|\alpha|}_{h},$$
	and let $\lambda$ satisfy
		\begin{equation}\label{firstasslambda}
	|\partial_\xi^\alpha \lambda(x,\xi)| \leq \rho_0 A^{|\alpha|} \alpha!^{\kappa} \langle \xi \rangle^{\frac{1}{\kappa}-|\alpha|}_{h}
	\end{equation}
	and \eqref{equation_stronger_hypothesis_on_Lambda} for $\beta \neq 0$. 
	Then there are $\tilde{\delta} > 0$ and $h_0 = h_0(A) \geq 1$ such that if $\rho_0 \leq \tilde{\delta} A^{-\frac{1}{\kappa}}$ and $h \geq h_0$, then 
	\begin{multline}\label{asymptotic_expansion}
	e^\lambda(x,D) p(x,D) ^{R}\{e^{-\lambda}(x,D)\} = p(x,D)+ \textrm{op} \left( \sum_{1 \leq |\alpha+\beta| < N} \frac{1}{\alpha!\beta!} \partial^{\alpha}_{\xi} \{\partial^{\beta}_{\xi} e^{\lambda(x,\xi)} D^{\beta}_{x}p(x,\xi) D^{\alpha}_{x}e^{-\Lambda(x,\xi)} \} \right)  \\ 
	+ r_{N}(x,D) + r_{\infty}(x,D) ,
	\end{multline}
	where 
	\begin{equation*}
		|\partial^{\alpha}_{\xi}\partial^{\beta}_{x}r_N(x,\xi)| \leq C_{\rho_0,A,\kappa} (C_{\kappa}A)^{|\alpha+\beta|+2N}\alpha!^{\kappa}\beta!^{\kappa}N!^{2\kappa-1} \langle \xi \rangle_h^{m-(1-\frac{1}{\kappa})N - |\alpha|},
	\end{equation*}
	\begin{equation*}
		|\partial^{\alpha}_{\xi}\partial^{\beta}_{x}r_{\infty}(x,\xi)| \leq C_{\rho_0,A,\kappa} (C_{\kappa}A)^{|\alpha+\beta|+2N}\alpha!^{\kappa}\beta!^{\kappa}N!^{2\kappa-1} e^{-c_\kappa A^{-\frac{1}{\kappa}} \langle \xi \rangle_h^{\frac{1}{\kappa}}}.
	\end{equation*}
In particular, $r_\infty(x,D)$ is $\kappa$-regularizing.
\end{theorem}

\begin{remark} \label{remark6}
 Notice that the operator $e^{\lambda}(x,D) p(x,D) ^{R}\{e^{-\lambda}(x,D)\}$ has been already treated in \cite[Theorem 6.12]{KN} (see also \cite[Theorem 2.1]{NT}). Namely, Theorem \ref{conjthmnew} and Theorem 6.12 in \cite{KN} only differ in the form of the asymptotic expansion of the symbol of $e^{\lambda}(x,D) p(x,D) ^{R}\{e^{-\lambda}(x,D)\}$.
 Nevertheless, having an asymptotic expansion of the form \eqref{asymptotic_expansion} is crucial to perform the computations in Section \ref{section_conjugation_of_iP} and formula \eqref{asymptotic_expansion} cannot be easily derived from the statement of \cite[Theorem 6.12]{KN}. This is the reason why we prefer to state the result in the form above. On the other hand, the proof of Theorem \ref{conjthmnew} follows the same argument of the proof of Theorem 6.12 in \cite{KN} and the different forms of the asymptotic expansions are obtained just by applying Taylor's formula at different stages of the proofs. For this reason, in the Appendix at the end of the paper we propose just a sketch of the proof of Theorem \ref{conjthmnew} where we just emphasize the main differences from the proof of Theorem 6.12 in \cite{KN}. 
\end{remark}


\section{Change of variables} \label{changeofvariables}
In this section we introduce the main ingredients which define the change of variable that we need for the analysis of  problem \eqref{Cauchy_problem_in_introduction}. \\
For $M_2, M_1 > 0$ and $h \geq 1$ a large parameter, we define 
\begin{equation}\label{equation_definition_lambda_2}
	\lambda_2(x, \xi) = M_2 w\left(\frac{\xi}{h}\right) \int_{0}^{x} \langle y \rangle^{-\sigma}  
	\psi\left(\frac{\langle y \rangle}{\langle \xi \rangle^{2}_{h}}\right) dy, \quad (x, \xi) \in \R^2,
\end{equation}
\begin{equation}\label{equation_definition_lambda_1}
	\lambda_1(x, \xi) = M_1 w\left(\frac{\xi}{h}\right) \langle \xi \rangle^{-1}_{h} \int_{0}^{x} \langle y \rangle ^{-\frac{\sigma}{2}} \psi\left(\frac{\langle y \rangle}{\langle \xi \rangle^{2}_{h}}\right) dy, \quad (x, \xi) \in \R^2,
\end{equation}
where
$$
w (\xi) =
\begin{cases}
0, \qquad \qquad \qquad \quad \,\,\, |\xi| \leq 1,\\
-\text{sgn}(\partial_{\xi}a_3(t, \xi)) , \quad |\xi| > R_{a_3},
\end{cases}
\quad 
\psi (y) =
\begin{cases}
1, \quad |y| \leq \frac{1}{2}, \\
0, \quad |y| \geq 1,
\end{cases}
$$	
$|\partial^{\alpha}_{\xi} w(\xi)| \leq C_{w}^{\alpha + 1} \alpha!^{\mu}$, $|\partial^{\beta}_{y} \psi(y)| \leq C_{\psi}^{\beta + 1}\beta!^{\mu}$, with $\mu > 1$.
Notice that thanks to assumption (i) in Theorem \ref{mainthm}, the function $w$ is constant for $\xi \geq R_{a_3}$ and for $\xi \leq -R_{a_3}$.

\begin{lemma}\label{lemma_estimates_lambda_2}
	Let $\lambda_2(x, \xi)$ as in (\ref{equation_definition_lambda_2}). Then the following estimates hold:
	\begin{itemize}
		\item[(i)] $|\lambda_2(x, \xi)| \leq \frac{M_2}{1-\sigma} \langle \xi \rangle^{2(1-\sigma)}_{h}$;
		\item[(ii)] $|\partial^{\alpha}_{\xi}\lambda_2(x, \xi)| \leq C^{\alpha+1} \alpha!^{\mu} \langle \xi \rangle^{2(1-\sigma)-\alpha}_{h}$, for $\alpha \geq 1$;
		\item[(iii)] $|\partial^{\alpha}_{\xi}\lambda_2(x, \xi)| \leq C^{\alpha+1} \alpha!^{\mu} \langle \xi \rangle^{-\alpha}_h \langle x \rangle^{1-\sigma}$, for $\alpha \geq 0$;
		\item[(iv)] $| \partial^{\alpha}_{\xi} \partial^{\beta}_{x}\lambda_2(x, \xi)| \leq C^{\alpha+\beta +1} \alpha!^{\mu} \beta!^{\mu} 
		\langle \xi \rangle^{-\alpha}_{h} \langle x \rangle^{-\sigma -(\beta-1)}$, for $\alpha \geq 0, \beta \geq 1$.
	\end{itemize}
\end{lemma}

\begin{proof}
	We denote by $\chi_{\xi}(x)$ the characteristic function of the set $\{ x \in \R : \langle x \rangle \leq \langle \xi \rangle^{2}_{h} \}$. Now note that 
	\begin{align*}
		|\lambda_2(x,\xi)| \leq M_2 \int_0^{|x|} \langle y \rangle^{-\sigma} \chi_{\xi}(y) dy  = M_2 \int^{\min\{|x|, \langle \xi \rangle^{2}_{h}\}}_{0} \langle y \rangle^{-\sigma} dy,
	\end{align*}
	hence 
	$$
	|\lambda_2(x,\xi)| \leq \frac{M_2}{1-\sigma} \min\{ \langle \xi \rangle^{2(1-\sigma)}_{h}, \langle x \rangle^{1-\sigma}\}.
	$$
	
	For $\alpha \geq 1$ we have 
	\begin{align*}
		|\partial^{\alpha}_{\xi} \lambda_2(x,\xi)| &\leq M_2 \sum_{\alpha_1 + \alpha_2 = \alpha} \frac{\alpha!}{\alpha_1!\alpha_2!} 
		\left| w^{(\alpha_1)}\left( \frac{\xi}{h} \right) \right| h^{-\alpha_1} \left| \int^{x}_{0} \chi_{\xi}(y) \langle y \rangle^{-\sigma} \partial^{\alpha_2}_{\xi} 
		\psi \left( \frac{\langle y \rangle}{\langle \xi \rangle^{2}_{h}} \right) \right| \, dy \\
		&\leq M_2 \sum_{\alpha_1 + \alpha_2 = \alpha} \frac{\alpha!}{\alpha_1!\alpha_2!} \left| w^{(\alpha_1)}\left( \frac{\xi}{h} \right) \right| h^{-\alpha_1}
		\int^{|x|}_{0} \chi_{\xi}(y) \langle y \rangle^{-\sigma} \\
		&\times \sum_{j=1}^{\alpha_2} \frac{ \left|\psi^{(j)} \left( \frac{\langle y \rangle}{\langle \xi \rangle^{2}_{h}} \right) \right| }{j!} \sum_{\gamma_1 + \ldots + \gamma_j = \alpha_2} \frac{\alpha_2!}{\gamma_1!\ldots \gamma_j!} \prod_{\nu = 1}^{j} \langle y \rangle |\partial^{\gamma_{\nu}}_{\xi} \langle \xi \rangle^{-2}_{h}| dy	\\
		&\leq M_2 \sum_{\alpha_1 + \alpha_2 = \alpha} \frac{\alpha!}{\alpha_1!\alpha_2!} C_{w}^{\alpha_1+1}\alpha_1!^{\mu} \langle \xi \rangle^{-\alpha_1}_{h} \langle R_{a_3} \rangle^{\alpha_1} \\
		&\times \tilde{C}_{\psi}^{\alpha_2 + 1} \alpha_2!^{\mu} \langle \xi \rangle_{h}^{-\alpha_2} 
		\int^{|x|}_{0} \chi_{\xi}(y) \langle y \rangle^{-\sigma}  dy  \\
		&\leq M_2 C_{w,\psi, R_{a_3}}^{\alpha+1} \alpha!^{\mu} \langle \xi \rangle^{-\alpha}_{h} \int^{|x|}_{0} \chi_{\xi}(y) \langle y \rangle^{-\sigma}  dy	 \\
		&\leq \frac{M_2}{1-\sigma} C_{w,\psi, R_{a_3}}^{\alpha+1} \alpha!^{\mu} \langle \xi \rangle^{-\alpha}_{h}
		\min \{ \langle \xi \rangle^{2(1-\sigma)}_{h}, \langle x \rangle^{1-\sigma} \}.
	\end{align*}

	For $\alpha \geq 0$ and $\beta \geq 1$ we have
	\begin{align*}
		|\partial^{\alpha}_{\xi} \partial^{\beta}_{x} &\lambda_2(x,\xi)| \leq 
		M_2 \sum_{\overset{\alpha_1+\alpha_2 = \alpha}{\beta_1 + \beta_2 = \beta-1}} \frac{\alpha!}{\alpha_1!\alpha_2!} \frac{(\beta-1)!}{\beta_1!\beta_2!}
		\left| w^{(\alpha_1)}\left( \frac{\xi}{h} \right) \right| |\partial^{\beta_{1}}_{x} \langle x \rangle^{-\sigma}| 
		\left| \partial^{\alpha_2}_{\xi} \partial^{\beta_2}_{x} \psi\left( \frac{\langle x \rangle }{ \langle \xi \rangle^{2}_{h} } \right) \right| \\
		&\leq M_2 \sum_{\overset{\alpha_1+\alpha_2 = \alpha}{\beta_1 + \beta_2 = \beta-1}} \frac{\alpha!}{\alpha_1!\alpha_2!} \frac{(\beta-1)!}{\beta_1!\beta_2!}
		C_{w}^{\alpha_1+1} \alpha_1!^{\mu} \langle R_{a_3} \rangle^{\alpha_1} \langle \xi \rangle^{-\alpha_1}_{h} C^{\beta_1} \beta_1! \langle x \rangle^{-\sigma - \beta_1} \\
		&\times \chi_{\xi}(x) \sum_{j = 1}^{\alpha_2 + \beta_2} \frac{\left| \psi^{(j)} \left( \frac{\langle x \rangle}{\langle \xi \rangle^{2}_{h}} \right) \right| }{ j! } \sum_{\overset{\gamma_1 + \ldots + \gamma_j = \alpha_2}{\lambda_1 + \ldots + \lambda_j = \beta_2}} \frac{\alpha_2!\beta_2!}{\gamma_1!\lambda_1! \ldots \gamma_j!\lambda_j!} \prod_{\nu = 1}^{j} 
		|\partial^{\lambda_{\nu}}_{x} \langle x \rangle| |\partial^{\gamma_\nu}_{\xi} \langle \xi \rangle^{-2}_{h}| \\
		&\leq M_2 \sum_{\overset{\alpha_1+\alpha_2 = \alpha}{\beta_1 + \beta_2 = \beta-1}} \frac{\alpha!}{\alpha_1!\alpha_2!} \frac{(\beta-1)!}{\beta_1!\beta_2!}
		C_{w}^{\alpha_1+1} \alpha_1!^{\mu} \langle R_{a_3} \rangle^{\alpha_1} \langle \xi \rangle^{-\alpha_1}_{h} C^{\beta_1} \beta_1! \langle x \rangle^{-\sigma - \beta_1} \\
		&\times \tilde{C}_{\psi}^{\alpha_2+\beta_2+1} \alpha_2!^{\mu} \beta_2!^{\mu} \langle x \rangle^{-\beta_2} \langle \xi \rangle^{-\alpha_2}_{h} \\
		&\leq M_2 C_{\psi, w, R_{a_3}}^{\alpha+\beta + 1} \alpha!^{\mu} (\beta-1)!^{\mu} \langle x \rangle^{-\sigma - (\beta-1)} \langle \xi \rangle^{-\alpha}_{h}.
	\end{align*}
\end{proof}

For $\lambda_1$ we have the following estimates which can be proved via the same arguments used for $\lambda_2$. We omit the proof for the sake of brevity.

\begin{lemma}\label{lemma_estimates_lambda_1}
	Let $\lambda_1(x,\xi)$ as in (\ref{equation_definition_lambda_1}). Then
	\begin{itemize}
		\item[(i)] $|\lambda_1(x, \xi)| \leq \frac{M_1}{1-\frac{\sigma}{2}} \langle \xi \rangle^{1-\sigma}_{h}$;
		\item[(ii)] $|\partial^{\alpha}_{\xi}\lambda_1(x, \xi)| \leq C^{\alpha+1} \alpha!^{\mu} \langle \xi \rangle^{1-\sigma - \alpha}_{h}$, for $\alpha \geq 1$;
		\item[(iii)]$|\partial^{\alpha}_{\xi}\lambda_1(x, \xi)| \leq C^{\alpha+1} \alpha!^{\mu} \langle \xi \rangle^{-1-|\alpha|}_h \langle x \rangle^{1-\frac{\sigma}{2} }$, for $\alpha \geq 0$;
		\item[(iv)]$|\partial^{\alpha}_{\xi}\lambda_1(x, \xi)| \leq C^{\alpha+1} \alpha!^{\mu} \langle \xi \rangle^{-|\alpha|}_h \langle x \rangle^{1-\sigma}$, for $\alpha \geq 0$;
		\item[(v)] $| \partial^{\alpha}_{\xi} \partial^{\beta}_{x}\lambda_1(x, \xi)| \leq C^{\alpha+\beta +1} \alpha!^{\mu} \beta!^{\mu} 
		\langle \xi \rangle^{-1-\alpha}_{h} \langle x \rangle^{-\frac{\sigma}{2} -(\beta-1)}$, for $\alpha \geq 0, \beta \geq 1$;
		\item[(vi)] $| \partial^{\alpha}_{\xi} \partial^{\beta}_{x}\lambda_1(x, \xi)| \leq C^{\alpha+\beta +1} \alpha!^{\mu} \beta!^{\mu} 
		\langle \xi \rangle^{-\alpha}_{h} \langle x \rangle^{-\sigma - (\beta-1)}$, for $\alpha \geq 0, \beta \geq 1$.
	\end{itemize}
\end{lemma}

\begin{remark}
 From Lemmas $\ref{lemma_estimates_lambda_2}$ and $\ref{lemma_estimates_lambda_1}$ we conclude $\lambda_2, \lambda_1 \in \textbf{\textrm{SG}}^{0, 1-\sigma}_{\mu}(\R^{2})$, $\lambda_1 \in S^{1-\sigma}_{\mu}(\R^{2})$ and $\lambda_2 \in S^{2(1-\sigma)}_{\mu}(\R^{2})$. Hence, setting $\tilde{\Lambda}=\lambda_1+\lambda_2$, we get $e^{\tilde{\Lambda}} \in S^\infty_{\mu; \frac1{2(1-\sigma)}}(\R^2) \cap \textbf{\textrm{SG}}^{0,\infty}_{\mu;\frac{1}{1-\sigma}}(\R^{2})$, cf. \cite{AAC3evolGelfand-Shilov}.  
\end{remark}

To construct the inverse of $e^{\tilde{\Lambda}}(x,D)$ we need to use the reverse operator $^{R}\{e^{-\tilde{\Lambda}}(x,D)\}$.
We have the following result which expresses the inverse of $e^{\tilde{\Lambda}}(x,D)$ 
in terms of composition of $^{R}\{e^{-\tilde{\Lambda}}(x,D)\}$ with a Neumann series.

\begin{lemma}\label{lemma_inverse_of_e_power_tilde_Lambda}
Let $\mu > 1$. For $h \geq 1$ large enough, the operator $e^{\tilde{\Lambda}}(x,D)$ is invertible 
and its inverse is given by 
$$
\{e^{\tilde{\Lambda}}(x,D)\}^{-1} = \hskip2pt ^R  \{e^{-\tilde{\Lambda}}(x,D)\} \circ \sum_{j \geq 0} (-r(x,D))^{j}, 
$$
for some $r = \tilde{r} + \bar{r}$, where $\tilde{r} \in \textbf{\textrm{SG}}^{-1,-\sigma}_{\mu}(\R^{2})$, $\bar{r} \in \Sigma_{\kappa}(\R^{2})$ for every 
$\kappa> 2\mu-1$ and 
$$
\tilde{r} - \sum_{1 \leq \gamma \leq N} \frac{1}{\gamma!} \partial^{\gamma}_{\xi}(e^{\tilde{\Lambda}} D^{\gamma}_{x} e^{-\tilde{\Lambda}}) \in\textbf{\textrm{SG}}^{-1-N,-\sigma-\sigma N}_{\mu}(\R^{2}), \quad \forall N \geq 1.
$$ 
Moreover, $\sum (-r(x,D))^{j}$ has symbol in $\textbf{\textrm{SG}}^{0,0}_{\mu}(\R^{2}) + \Sigma_{\kappa}(\R^{2})$ for every $\kappa >2\mu-1$. Finally, we have	
\begin{equation}\label{equation_inverse_of_e_power_tilde_Lambda_in_a_precise_way}
	\{e^{\tilde{\Lambda}}(x,D) \}^{-1} = \hskip2pt ^R \{e^{-\tilde{\Lambda}}(x,D)\} \circ \textrm{op} ( 1 - i\partial_{\xi} \partial_{x} \tilde{\Lambda} - \frac{1}{2}\partial^{2}_{\xi}(\partial^{2}_{x}\tilde{\Lambda} - [\partial_x \tilde{\Lambda}]^{2}) - [\partial_{\xi} \partial_{x} \tilde{\Lambda} ]^{2} + q_{-3} ),
\end{equation}
where $q_{-3} \in \textbf{\textrm{SG}}^{-3,-3\sigma}_{\mu}(\R^{2}) + \Sigma_{\kappa}(\R^{2})$.
\end{lemma}

\begin{proof}
	The proof follows directly from \cite[Lemma 4]{AAC3evolGelfand-Shilov}. 
\end{proof}

\begin{remark}
	Since we can choose $\mu > 1$ arbitrarily close to $1$, we may assume $2\mu-1 < \theta$. Therefore we can take $\kappa < \theta$ in the above lemma.  
\end{remark}


\section{Sobolev well-posedness for the Cauchy problem \eqref{equation_conjugated_cauchy_problem}
}
\label{section_conjugation_of_iP}

In this section we will perform the conjugation of $iP$ in \eqref{diffP},\eqref{Cauchy_problem_in_introduction} by the operator $e^{\Lambda}(t, x, D_x)$ defined by \eqref{nuovadefLambda}, where $k\in C^1([0,T];\R)$ is a positive non increasing function and $\tilde{\Lambda}= \lambda_1+\lambda_2,$ with $\lambda_1,\lambda_2$ defined by \eqref{equation_definition_lambda_1}, \eqref{equation_definition_lambda_2} respectively. Namely, we explicitly compute the operator $P_\Lambda$ in \eqref{equation_conjugated_cauchy_problem} and prove that the Cauchy problem \eqref{equation_conjugated_cauchy_problem} is well-posed in Sobolev spaces $H^m$, $m\in\R$. For this purpose we shall use Theorem \ref{conjthmnew}.
\\ Before performing the conjugation, let us make some remarks. By Lemmas \ref{lemma_estimates_lambda_2} and \ref{lemma_estimates_lambda_1} we get 
$$
|\partial^{\alpha}_{\xi} \partial^{\beta}_{x} \tilde{\Lambda} (x,\xi)| \leq 
\begin{cases}
	C_{\tilde{\Lambda}}^{|\alpha+\beta|+1} \alpha!^{\mu} \beta!^{\mu} \langle \xi \rangle^{2(1-\sigma) - \alpha}_{h}, \\
	C_{\tilde{\Lambda}}^{|\alpha+\beta|+1} \alpha!^{\mu} \beta!^{\mu} \langle \xi \rangle^{- \alpha}_{h}, \,\, \text{if} \,\, \beta > 0,
\end{cases}
$$
where $C_{\tilde{\Lambda}}$ is a constant depending only on $M_2, M_1, C_{w}, C_{\psi}, \mu, \sigma$. Moreover, since we are assuming $2(1-\sigma) < \frac{1}{\theta}$ we also get
\begin{align*}
	|\tilde{\Lambda}(x,\xi)| &\leq C_{\tilde{\Lambda}} \langle \xi \rangle^{2(1-\sigma)}_{h} = C_{\tilde{\Lambda}} \langle \xi \rangle^{2(1-\sigma) - \frac{1}{\theta}}_{h} \langle \xi \rangle^{\frac{1}{\theta}}_{h} \leq h^{\frac{1}{\theta} - 2(1-\sigma)} C_{\tilde{\Lambda}} \langle \xi \rangle^{\frac{1}{\theta}}_{h},
\end{align*}
therefore
$$
\sup_{x,\xi \in \R^{n}} |\tilde{\Lambda}(x,\xi)| \langle \xi \rangle^{-\frac{1}{\theta}}_{h} \leq h^{\frac{1}{\theta} - 2(1-\sigma)}C_{\tilde{\Lambda}},
$$
which can be assumed as small as we want, provided that $h > h_0(M_2,M_1, C_{\tilde{\Lambda}}, \theta, \sigma)$. Hence we may use Theorem \ref{conjthmnew} to compute $e^{\tilde{\Lambda}}(x,D) \circ (iP) \circ \{e^{\tilde{\Lambda}}(x,D)\}^{-1}$. First we note that $e^{\tilde{\Lambda}(x,D)} \circ \partial_t  \circ \{e^{\tilde{\Lambda}}(x,D)\}^{-1} =\partial_t,$ because $\tilde{\Lambda}(x,\xi)= \lambda_1(x,\xi)+\lambda_2(x,\xi)$ is independent of $t$.\\

	\begin{itemize} \item Conjugation of $ia_3(t,D)$: Since $a_3$ does no depend on $x$, the asymptotic expansion of Theorem \ref{conjthmnew} reduces to (omitting $(t,x,D)$ in the notation)
	$$
	e^{\tilde{\Lambda}}\circ a_3 \circ \, ^{R}\{e^{-\tilde{\Lambda}}\} = a_3+\textrm{op}\left( \sum_{1 \leq \alpha < N} \frac{1}{\alpha!} \partial^{\alpha}_{\xi} \{ e^{\tilde{\Lambda}} a D^{\alpha}_{x} e^{-\tilde{\Lambda}}\} + r_N\right) + r_\infty,
	$$
and since the $x-$derivatives of $\tilde{\Lambda}$ kill the growth in $\xi$ given by the integrals defining $\tilde\Lambda$, we obtain $r_N$ of order $3 - N$ and $r_\infty\in\mathcal K_\theta$; with this simplification we get
\begin{align*}
	e^{\tilde{\Lambda}} \circ ia_3 \circ\, ^{R}\{e^{-\tilde{\Lambda}}\} = ia_3 + \partial_{\xi}\{ia_3 D_x(-\tilde{\Lambda})\} + \frac{1}{2} \partial^{2}_{\xi} \{ia_3 [D^{2}_{x}(-\tilde{\Lambda}) +(D_x\tilde{\Lambda})^{2}]\} + r_{3} + r_{\infty}
\end{align*}
 with $r_{3}$ of order zero. Composing with the Neumman series we get from \eqref{equation_inverse_of_e_power_tilde_Lambda_in_a_precise_way}:

\begin{align*}
	e^{\tilde{\Lambda}} (ia_3) &\{e^{\tilde{\Lambda}}\}^{-1} 
	= \textrm{op}\left(ia_3 - \partial_{\xi}(a_3 \partial_{x}\tilde{\Lambda}) + \frac{i}{2}\partial^{2}_{\xi}[a_3(\partial^{2}_{x}\tilde{\Lambda} - (\partial_{x}\tilde{\Lambda})^2)] + r_3 + r_{\infty}\right) \\ 
	&\circ \textrm{op}\left( 1 - i\partial_{\xi} \partial_{x} \tilde{\Lambda} - \frac{1}{2}\partial^{2}_{\xi}(\partial^{2}_{x}\tilde{\Lambda} - [\partial_x \tilde{\Lambda}]^{2}) - [\partial_{\xi} \partial_{x} \tilde{\Lambda} ]^{2} + q_{-3} \right) \\
	&= ia_3 - \partial_{\xi}(a_3 \partial_{x}\tilde{\Lambda}) + \frac{i}{2}\partial^{2}_{\xi}\{a_3(\partial^{2}_{x}\tilde{\Lambda} - \{\partial_{x}\tilde{\Lambda}\}^2)\}
	+ a_3\partial_{\xi}\partial_{x}\tilde{\Lambda} - i\partial_{\xi}a_3\partial_{\xi}\partial^{2}_{x}\tilde{\Lambda} 
	\\
	&+ i\partial_{\xi}(a_3 \partial_{x}\tilde{\Lambda})\partial_{\xi}\partial_{x}\tilde{\Lambda} - \frac{i}{2}a_3\{\partial^{2}_{\xi}(\partial^{2}_{x}\tilde{\Lambda} + [\partial_x \tilde{\Lambda}]^{2}) + 2[\partial_{\xi} \partial_{x} \tilde{\Lambda} ]^{2}\} + r_0 + r \\
	&= ia_3 - \partial_{\xi}a_3 \partial_{x}\tilde{\Lambda} + \frac{i}{2}\partial^{2}_{\xi}\{a_3[\partial^{2}_{x}\tilde{\Lambda} - 
	(\partial_{x}\tilde{\Lambda})^2]\}
	- i\partial_{\xi}a_3\partial_{\xi}\partial^{2}_{x}\tilde{\Lambda} 
	\\
	&+ i\partial_{\xi}(a_3 \partial_{x}\tilde{\Lambda})\partial_{\xi}\partial_{x}\tilde{\Lambda} - \frac{i}{2}a_3\{\partial^{2}_{\xi}(\partial^{2}_{x}\tilde{\Lambda} + [\partial_x \tilde{\Lambda}]^{2}) + 2(\partial_{\xi} \partial_{x} \tilde{\Lambda} )^{2}\} + r_0 + r, \\
\end{align*}
where $r_0 \in C([0,T]; S^{0}_{\mu'}(\R^{2}))$ and $r$ is a new regularizing term. Writing $\tilde{\Lambda} = \lambda_2 + \lambda_1$ and observing that $D_x \lambda_1$ has order $-1$ we get
\begin{align*}
	e^{\tilde{\Lambda}} (ia_3) \{e^{\tilde{\Lambda}}\}^{-1} &= ia_3 - \partial_{\xi}a_3 \partial_{x}\lambda_2 - \partial_{\xi}a_3\partial_{x}\lambda_1 + \frac{i}{2}\partial^{2}_{\xi}\{a_3(\partial^{2}_{x}\lambda_2 - \{\partial_{x}\lambda_2\}^2)\}
	- i\partial_{\xi}a_3\partial_{\xi}\partial^{2}_{x}\lambda_2 
	\\
	&+ i\partial_{\xi}(a_3 \partial_{x}\lambda_2)\partial_{\xi}\partial_{x}\lambda_2 - \frac{i}{2}a_3\{\partial^{2}_{\xi}(\partial^{2}_{x}\lambda_2 + [\partial_x \lambda_2]^{2}) + 2[\partial_{\xi} \partial_{x} \lambda_2 ]^{2}\} + r_0 + r,
\end{align*}
for a new zero order term $r_0$. Setting 
\begin{align*}
	d_1(t,x,\xi) &= \frac{1}{2}\partial^{2}_{\xi}\{a_3(\partial^{2}_{x}\lambda_2 - \{\partial_{x}\lambda_2\}^2)\}
	- \partial_{\xi}a_3\partial_{\xi}\partial^{2}_{x}\lambda_2 \\
	&+ \partial_{\xi}(a_3 \partial_{x}\lambda_2)\partial_{\xi}\partial_{x}\lambda_2 - \frac{1}{2}a_3\{\partial^{2}_{\xi}(\partial^{2}_{x}\lambda_2 + [\partial_x \lambda_2]^{2}) + 2[\partial_{\xi} \partial_{x} \lambda_2 ]^{2}\}
\end{align*}
we may write
\begin{align*}
	e^{\tilde{\Lambda}} (ia_3) \{e^{\tilde{\Lambda}}\}^{-1} &= ia_3 - \partial_{\xi}a_3 \partial_{x}\lambda_2 - \partial_{\xi}a_3\partial_{x}\lambda_1 + id_1 + r_0 + r,
\end{align*}
where $d_1$ is a real valued symbol of order $1$ which does not depend on $\lambda_1$. Moreover, we have the following estimate
$$
|\partial^{\alpha}_{\xi} \partial^{\beta}_{x} d_1(t,x,\xi)| \leq C_{\lambda_2}^{\alpha+\beta+1} \alpha!^{\mu} \beta!^{\mu} \langle \xi \rangle^{1 - \alpha}\langle x \rangle^{-\sigma};
$$
where $C_{\lambda_2}$ is a constant dependent of $\lambda_1$.

\item Conjugation of $ia_2(t,x,D)$: for $N \in \N$ such that $2-N(1-\frac{1}{\theta}) \leq 0$, Theorem \ref{conjthmnew}  gives 
\begin{align*}
	e^{\tilde{\Lambda}} \circ ia_2(t,x,D) \circ\, ^{R}\{e^{-\tilde{\Lambda}}\} = ia_2(t,x,D) &+ 
	\text{op}\underbrace{ \left( \sum_{1 \leq \alpha + \beta < N} \frac{1}{\alpha! \beta!} \partial^{\alpha}_{\xi} \{ \partial^{\beta}_{\xi} e^{\tilde{\Lambda}} D^{\beta}_{x} (ia_2) D^{\alpha}_{x} e^{-\tilde{\Lambda}}\} \right)}_{=:(ia_2)_{N}}  \\
	&+ \tilde{r}_0(t,x,D) + \tilde{r}(t,x,D),	
\end{align*}
where $\tilde{r}_0$ has order zero and $\tilde r\in\mathcal K_\theta$. By the hypothesis on $a_2$, we obtain
$$
|\partial^{\alpha}_{\xi} \partial^{\beta}_{x} (ia_2)_{N}(t,x,\xi)| \leq C_{a_2, \tilde{\Lambda}}^{\alpha+\beta+1} \alpha!^{\mu} \beta!^{s_0} \langle \xi \rangle^{2 - [2\sigma-1] - \alpha} \langle x \rangle^{-\sigma}.
$$
Composing with the Neumann series and using the fact that $\partial_x \lambda_1$ has order $-1$ we get
\begin{align*}
	e^{\tilde{\Lambda}} \circ ia_2 \circ \{e^{\tilde{\Lambda}}\}^{-1} &= (ia_2 + (ia_2)_{N} + \tilde{r}_0 + \tilde{R})\circ(I - i\partial_\xi\partial_x \lambda_2 + q_{-2}) \\
	&= ia_2 + (ia_2)_{N} + a_2 \circ \partial_\xi\partial_x\lambda_2 - i(ia_2)_{N} \circ\partial_\xi\partial_x \lambda_2 + r_0 + r \\
	&= ia_2 + \underbrace{(ia_2)_{N} - i(ia_2)_{N} \partial_\xi\partial_x \lambda_2}_{=: (ia_2)_{\tilde{\Lambda}} } + a_2 \partial_\xi\partial_x\lambda_2 + r_0 + r,
\end{align*}
where $r_0$ has order zero, $r\in\mathcal K_\theta$ and $(a_2)_{\tilde{\Lambda}}$ satisfies 
$$
|\partial^{\alpha}_{\xi} \partial^{\beta}_{x} (ia_2)_{\tilde{\Lambda}}(t,x,\xi)| \leq C_{a_2, \tilde{\Lambda}}^{\alpha+\beta+1} \alpha!^{\mu} \beta!^{s_0} \langle \xi \rangle^{2 - [2\sigma-1] - \alpha} \langle x \rangle^{-\sigma},
$$
in particular
\begin{equation}\label{equation_remainder_of_conj_of_a_2_by_the_reverse}
	|(ia_2)_{\tilde{\Lambda}}(t,x,\xi)| \leq C_{a_2, \tilde{\Lambda}} \langle \xi \rangle^{2 - [2\sigma-1]}_{h} \langle x \rangle^{-\sigma}.
\end{equation}

\item Conjugation of $ia_1(t,x,D)$:
\begin{align*}
	e^{\tilde{\Lambda}} \circ (ia_1)(t,x,D) \circ \{e^{\tilde{\Lambda}}\}^{-1} &= (ia_1 + (ia_1)_{\tilde{\Lambda}} + r_1) (t,x,D) \sum_{j\geq 0}(-r)^{j} \\
	&= ia_1(t,x,D) + (ia_1)_{\tilde{\Lambda}}(t,x,D) + r_{0}(t,x,D) + r(t,x,D), 
\end{align*}
where $r_{0}$ has order zero, $r\in\mathcal K_\theta$ and 
\begin{equation}\label{equation_remainder_of_conj_of_a_1_by_the_reverse}
	(ia_{1})_{\tilde{\Lambda}} \sim \sum_{|\alpha+\beta| \geq 1} \frac{1}{\alpha!\beta!} \partial^{\alpha}_{\xi}\{\partial^{\beta}_{\xi}e^{\tilde{\Lambda}} D^{\beta}_{x}(ia_1) D^{\alpha}_{x}e^{-\tilde{\Lambda}}\} \,\, \text{in} \,\, S^{2(1-\sigma)}_{\mu, s_0}.
\end{equation}

\item Conjugation of $ia_0(t,x,D)$: $e^{\tilde{\Lambda}} \circ (ia_0)(t,x,D) \circ \{e^{\tilde{\Lambda}}\}^{-1} = r_0(t,x,D) + r(t,x,D),$ where $r_0$ has order zero and $r$ is a $\theta-$regularizing term.

\end{itemize}

Gathering all the previous computations we may write (omitting $(t,x,D)$ in the notation)
\begin{align*}
e^{\tilde{\Lambda}} &(iP) \{e^{\tilde{\Lambda}}\}^{-1} = \partial_{t} + ia_3 - \partial_{\xi}a_3\partial_x\lambda_2 - \partial_{\xi}a_3 \partial_{x}\lambda_1 + id_{1} \\
&+ ia_2 + (ia_2)_{\tilde{\Lambda}} + a_2\partial_{\xi}\partial_{x}\lambda_2 + ia_1 + (ia_1)_{\tilde{\Lambda}} + r_0 + r,
\end{align*}
where $d_{1} \in S^{1}_{1,s_0}$, $d_{1}$ is real-valued, $d_{1}$ does not depend on $\lambda_1$, $(ia_2)_{\tilde{\Lambda}}$ satisfies \eqref{equation_remainder_of_conj_of_a_2_by_the_reverse}, $(ia_1)_{\tilde{\Lambda}}$ satisfies \eqref{equation_remainder_of_conj_of_a_1_by_the_reverse}, $r_0 \in C([0,T]; S^{0}_{\mu', s_0}(\R^2))$ and $r\in\mathcal K_\theta$.

\subsection{Conjugation of $e^{\tilde{\Lambda}} (iP) \{e^{\tilde{\Lambda}}\}^{-1}$ by $e^{k(t)\langle D \rangle^{\frac{1}{\theta}}_{h}}$}

Let us recall that the function $k(t)$ satisfies $k \in C^{1}([0,T]; \R)$, $k'(t)\leq 0$ and $k(t) > 0$ for every $t\in [0,T]$. In order to apply Theorem \ref{conjthmnew} with $\lambda(t,x,\xi)= k(t) \langle \xi \rangle_h^{\frac1{\theta}}$, we observe that this function satisfies \eqref{equation_stronger_hypothesis_on_Lambda} and \eqref{firstasslambda} with $\rho_0=k(0)$ and for some positive $A$. Hence, there exists $\tilde{\delta}>0$ such that if $k(0) <\tilde{\delta} A^{-1/\theta},$ then Theorem \ref{conjthmnew} applies. Moreover, since in this case $\lambda$ does not depend on $x$, the asymptotic expansion of Theorem \ref{conjthmnew} simplifies into
\begin{multline}\label{z}
	e^{k(t)\langle D \rangle_h^{1/\theta}} p(x,D) e^{-k(t)\langle D \rangle_h^{1/\theta}} = p(x,D)+ \textrm{op} \left( \sum_{1 \leq |\beta| < N} \frac{1}{\beta!} \partial^{\beta}_{\xi} e^{k(t)\langle \xi \rangle_h^{1/\theta}} D^{\beta}_{x}p(x,\xi) e^{-k(t)\langle \xi \rangle_h^{1/\theta}}  \right)  \\ + r_{N}(x,D) + r_{\infty}(x,D),
	\end{multline}
where we can say that $r_{N}+ r_{\infty} \in \tilde S_\theta^{m-(1-\frac{1}{\theta})N}$.

\begin{itemize}
\item Conjugation of $ \partial_t$: $e^{k(t)\langle D \rangle^{\frac{1}{\theta}}_{h}  } \, \partial_{t} \, e^{-k(t)\langle D \rangle^{\frac{1}{\theta}}_{h}} = \partial_{t} - k'(t)\langle D \rangle^{\frac{1}{\theta}}_{h}$.

\item Conjugation of $ia_3(t,D)$: since $a_3$ does not depend on $x$, we simply have 
$$
e^{k(t)\langle D \rangle^{\frac{1}{\theta}}_{h}} \circ ia_3 (t,D) \circ e^{-k(t)\langle D \rangle^{\frac{1}{\theta}}_{h}} = ia_3(t,D).
$$

\item Conjugation of $\text{op}\{ia_2 - \partial_{\xi}a_3\partial_{x}\lambda_2\}$:
\begin{align*}
	e^{k(t)\langle D \rangle^{\frac{1}{\theta}}_{h}} \circ (ia_2 - \partial_{\xi}a_3\partial_{x}\lambda_2 ) &(t,x,D) \circ e^{-k(t)\langle D \rangle^{\frac{1}{\theta}}_{h}} = 
	ia_2(t,x,D) \\ 
	&- \text{op}(\partial_{\xi}a_3\partial_{x}\lambda_2) + (b_{2,k} + r_0 )(t,x,D)
\end{align*}
where $r_0$ has order zero 
and $b_{2,k} (t,x,\xi) \in C([0,T]; S^{1+\frac{1}{\theta}}_{\mu,s_0}(\R^{2}))$,
\begin{equation}\label{equation_remainde_conj_level_2_by_k_D}
	|b_{2,k}(t,x,\xi)| \leq \max\{1,k(t)\}C_{s,\lambda_2} \langle \xi \rangle^{1+\frac{1}{\theta}}_{h} \langle x \rangle^{-\sigma}, \quad x, \xi \in \R^{n}.
\end{equation}

\item Conjugation of $(ia_2)_{\tilde{\Lambda}} (t,x,D)$: 
$$
e^{k(t)\langle D \rangle^{\frac{1}{\theta}}_{h}} \circ (ia_2)_{\tilde{\Lambda}}(t,x,D) \circ e^{-k(t)\langle D \rangle^{\frac{1}{\theta}}_{h}} = \{(ia_2)_{k, \tilde{\Lambda}} +r_0\}(t,x,D),
$$
where $r_0$ has order zero 
and $(ia_2)_{k, \tilde{\Lambda}} \in C([0,T]; S^{2-(2\sigma-1)}_{\mu, s_0})$, 
$$
|\partial^{\alpha}_{\xi} \partial^{\beta}_{x} (ia_2)_{k,\tilde{\Lambda}}(t, x,\xi)| \leq (\max\{k(t),1\}C_{a_2, \tilde{\Lambda}})^{\alpha+\beta+1} \alpha!^{\mu} \beta!^{s_0} \langle \xi \rangle^{2 - (2\sigma-1)- \alpha} \langle x \rangle^{-\sigma}.
$$
In particular
\begin{equation}\label{equation_remainder_conj_ia_2_Lambda_tilde_by_k_D}
	|(ia_2)_{k,\tilde{\Lambda}}(t, x,\xi)| \leq \max\{k(t),1\}C_{a_2, \tilde{\Lambda}} \langle \xi \rangle^{2 - (2\sigma - 1)}_{h} \langle x \rangle^{-\sigma}.
\end{equation}

\item Conjugation of $\text{op}(ia_1  - \partial_{\xi}a_3 \partial_{x}\lambda_1 + id_{1} + a_2\partial_{\xi}\partial_{x}\lambda_2)$: we have (omitting $(t,x,D)$ in the notation)
\begin{align*}
	e^{k(t)\langle D \rangle^{\frac{1}{\theta}}_{h}} &\circ \text{op}(ia_1  - \partial_{\xi}a_3 \partial_{x}\lambda_1 + id_{1} + a_2\partial_{\xi}\partial_{x}\lambda_2) \circ e^{-k(t)\langle D \rangle^{\frac{1}{\theta}}_{h}}  \\ 
	&= ia_1  - \partial_{\xi}a_3 \partial_{x}\lambda_1 + id_{1} + a_2\partial_{\xi}\partial_{x}\lambda_2+ b_{1,k} + r_0 ,
\end{align*}
where $r_0$  has order zero, $b_{1,k}(t,x,\xi) \in C([0,T]; S^{\frac{1}{\theta}}_{\mu,s_0})$ and for large $h$ we have
\begin{equation}\label{equation_remainder_conj_level_1_Lambda_tilde_by_k_D}
	|b_{1,k}(t,x,\xi)| \leq k(t)C_{\tilde{\Lambda}} \langle \xi \rangle^{\frac{1}{\theta}}_{h}, \quad x \in \R, \,\xi \in \R.
\end{equation}

\item Conjugation of $(ia_1)_{\tilde{\Lambda}}(t,x,D)$: 
$$
e^{k(t)\langle D \rangle^{\frac{1}{\theta}}_{h}} \circ (ia_1)_{\tilde{\Lambda}}(t,x,D) \circ e^{-k(t)\langle D \rangle^{\frac{1}{\theta}}_{h}} = \{(ia_1)_{k, \tilde{\Lambda}} + r_0 \}(t,x, D),
$$
where $r_0$ has order zero, $(ia_1)_{k, \tilde{\Lambda}} \in C ([0,T]; S^{2(1-\sigma)}_{\mu, s_0})$ and for large $h$ we have 
\begin{equation}\label{equation_remainder_conj_level_less_1_Lambda_tilde_by_k_D}
	|(ia_1)_{k,\tilde{\Lambda}}(t, x,\xi)| \leq C_{\tilde{\Lambda}} \langle \xi \rangle^{2(1-\sigma)}_{h}, \quad x, \xi \in \R.
\end{equation}

\end{itemize}

Finally, gathering all the previous computations we obtain the following expression for the conjugated opeartor (assuming the parameter $h$ sufficiently large)
\begin{align*}
e^{\Lambda} \circ (iP) \circ \{e^{\Lambda}\}^{-1} &= \partial_{t} + ia_3(t,D) \\
&+ \text{op}(ia_2  - \partial_{\xi}a_3\partial_x\lambda_2 + b_{2,k} + (ia_2)_{k,\tilde{\Lambda}}) \\
&+ \text{op}(ia_1 - \partial_{\xi}a_3 \partial_{x}\lambda_1 + id_{1}  + a_2\partial_{\xi}\partial_{x}\lambda_2) \\
&+ \text{op}(- k'(t)\langle \xi \rangle^{\frac{1}{\theta}}_{h} + b_{1,k}  + (ia_1)_{k, \tilde{\Lambda}}) + r_0(t,x,D)
\end{align*}
where $b_{2,k}$ satisfies \eqref{equation_remainde_conj_level_2_by_k_D}, $(ia_2)_{k, \tilde{\Lambda}}$ satisfies \eqref{equation_remainder_conj_ia_2_Lambda_tilde_by_k_D}, $b_{1,k}$ satisfies \eqref{equation_remainder_conj_level_1_Lambda_tilde_by_k_D}, $(ia_1)_{k, \tilde{\Lambda}}$ satisfies \eqref{equation_remainder_conj_level_less_1_Lambda_tilde_by_k_D} and $r_0$ has order zero.


\subsection{Lower bound estimates for the real parts}\label{section_estimate_for_the_real_parts}
In this subsection, we will derive some estimates from below for the real parts of the lower order terms of $(iP)_\Lambda$ and we use them to achieve a well-posedness result for the Cauchy problem \eqref{equation_conjugated_cauchy_problem}. We start noticing that for $|\xi| > hR_{a_3}$ we have
\begin{align*}
	- \partial_{\xi}a_3\partial_x\lambda_2 &= |\partial_{\xi}a_3| M_2 \langle x \rangle^{-\sigma}  
	\psi\left(\frac{\langle x \rangle}{\langle \xi \rangle^{2}_{h}}\right) \\
	&= |\partial_{\xi}a_3| M_2 \langle x \rangle^{-\sigma}  - |\partial_{\xi}a_3| M_2 \langle x \rangle^{-\sigma} \left[  1 - \psi\left(\frac{\langle x \rangle}{\langle \xi \rangle^{2}_{h}}\right) \right],
\end{align*}
\begin{align*}
	- \partial_{\xi}a_3\partial_x\lambda_1 &= |\partial_{\xi}a_3| M_1 \langle \xi \rangle^{-1}_{h} \langle x \rangle^{-\frac{\sigma}{2}}  
	\psi\left(\frac{\langle x \rangle}{\langle \xi \rangle^{2}_{h}}\right) \\ 
	&= |\partial_{\xi}a_3| M_1 \langle \xi \rangle^{-1}_{h} \langle x \rangle^{-\frac{\sigma}{2}}  - |\partial_{\xi}a_3| M_1 \langle \xi \rangle^{-1}_{h} \langle x \rangle^{-\frac{\sigma}{2}} \left[  1 - \psi\left(\frac{\langle x \rangle}{\langle \xi \rangle^{2}_{h}}\right) \right].
\end{align*}
We also observe that
$$
- |\partial_{\xi}a_3| M_2 \langle x \rangle^{-\sigma} \left[  1 - \psi\left(\frac{\langle x \rangle}{\langle \xi \rangle^{2}_{h}}\right) \right] \geq - 2^{\sigma}C_{a_3} M_2 \langle \xi \rangle^{2(1-\sigma)}_{h}, 
$$
$$
- |\partial_{\xi}a_3| M_1 \langle \xi \rangle^{-1}_{h} \langle x \rangle^{-\frac{\sigma}{2}} \left[  1 - \psi\left(\frac{\langle x \rangle}{\langle \xi \rangle^{2}_{h}}\right) \right] \geq - 2^{\sigma}C_{a_3}M_1 \langle \xi \rangle^{1-\sigma}_{h},
$$
because $\langle x \rangle \geq \frac12 \langle \xi \rangle^{2}_{h}$ on the support of $(1-\psi)(\langle x \rangle \langle \xi \rangle^{-2}_{h})$. 

In this way we may write 
$$
e^{\Lambda} \circ (iP)  \circ \{e^{\Lambda}\}^{-1} = \partial_{t} + ia_3(t,D) + \tilde{a}_2(t,x,D) + \tilde{a}_1(t,x,D) + \tilde{a}_{\theta}(t,x,D) + r_0(t,x,D),
$$
where $r_0$ is an operator of order $0$ and 
\beqsn
Re\, \tilde{a}_2 &=& -Im\, a_2 + |\partial_{\xi}a_3| M_2 \langle x \rangle^{-\sigma} + Re\, b_{2,k} + Re\, (ia_2)_{k,\tilde{\Lambda}},
\\
Im\, \tilde{a}_2 &=& Re \, a_2 + Im\, b_{2,k} + Im\, (a_2)_{k, \tilde{\Lambda}},
\\
Re\, \tilde{a}_1 &=& -Im\, a_1 + |\partial_{\xi}a_3| M_1 \langle \xi \rangle^{-1}_{h} \langle x \rangle^{-\frac{\sigma}{2}}  + Re\, a_2\partial_{\xi}\partial_{x}\lambda_2,
\\
Re\, \tilde{a}_\theta &=& - k'(t)\langle \xi \rangle^{\frac{1}{\theta}}_{h} + Re\, b_{1,k}  + Re\, (ia_1)_{k, \tilde{\Lambda}} \\  
&-& |\partial_{\xi}a_3| M_2 \langle x \rangle^{-\sigma} \left[  1 - \psi\left(\frac{\langle x \rangle}{\langle \xi \rangle^{2}_{h}}\right) \right] 
- |\partial_{\xi}a_3| M_1 \langle \xi \rangle^{-1}_{h} \langle x \rangle^{-\frac{\sigma}{2}} \left[  1 - \psi\left(\frac{\langle x \rangle}{\langle \xi \rangle^{2}_{h}}\right) \right].
\eeqsn

Now we decompose $iIm\, \tilde{a}_2$ into its Hermitian and anti-Hermitian part:
$$
i \Im \tilde a_2=\ds\frac{i \Im \tilde a_2+(i \Im \tilde a_2)^*}{2}+\frac{i \Im \tilde a_2-(i \Im \tilde a_2)^*}{2}= H_{Im\,\tilde{a}_2} + A_{Im\,\tilde{a}_2};
$$
we have that $2Re\, \langle A_{Im\,\tilde{a}_2} u, u \rangle = 0$, while $H_{Im\,\tilde{a}_2}$ has symbol 
$$
\sum_{\alpha\geq 1}\frac{i}{2\alpha!}\partial_\xi^\alpha D_x^\alpha \Im \tilde a_2(t,x,\xi) = 
\underbrace{\sum_{\alpha\geq 1}\frac{i}{2\alpha!}\partial_\xi^\alpha D_x^\alpha Re\, a_2}_{=:c(t,x,\xi)} + 
\underbrace{\sum_{\alpha\geq 1}\frac{i}{2\alpha!}\partial_\xi^\alpha D_x^\alpha \{Im\, b_{2,k} + Im\, (a_2)_{k, \tilde{\Lambda}}\}}_{=:e(t,x,\xi)}.
$$
The hypothesis on $a_2$ implies
$$
|c(t,x,\xi)| \leq C_{c} \langle \xi \rangle \langle x \rangle^{-\sigma},
$$
with $C_c$ depending only on $Re a_2$, whereas from \eqref{equation_remainde_conj_level_2_by_k_D}, \eqref{equation_remainder_conj_ia_2_Lambda_tilde_by_k_D} and using the fact that $2(1-\sigma) \leq \frac{1}{\theta}$ we obtain 
$$
|e(t,x,\xi)| \leq C_{e, k, \tilde{\Lambda}} \langle \xi \rangle^{\frac{1}{\theta}} \langle x \rangle^{-\sigma},
$$
with $C_{e, k. \tilde{\Lambda}}$ depending on $k(t)$ and $\tilde{\Lambda}$.

We are ready to obtain the desired estimates from below. Using the above decomposition we get
\begin{align*}
	e^{\Lambda} \circ (iP)  \circ \{e^{\Lambda}\}^{-1} = \partial_{t} &+ ia_3(t,D) + Re\,\tilde{a}_2(t,x,D) + A_{Im\,\tilde{a}_2}(t,x,D) \\ 
	&+ (\tilde{a}_1+c+e)(t,x,D) + \tilde{a}_{\theta}(t,x,D) + r_0(t,x,D).
\end{align*}
Note that $\langle \xi \rangle^{2}_{h} \leq 2 \xi^{2}$ provided that $|\xi| > R_{a_3}h$. Estimating the terms of order $2$ we get 
\begin{align} \label{lbestA_2}
	Re\, \tilde{a}_2 &\geq M_2 \frac{C_{a_3}}{2} \langle \xi \rangle^{2}_{h} \langle x \rangle^{-\sigma} - C_{a_2} \langle \xi \rangle^{2}_{h} \langle x \rangle^{-\sigma}  \langle x \rangle^{-\sigma}  \\
	&- \max\{1,k(t)\}C_{ \lambda_2} \langle \xi \rangle^{1+\frac{1}{\theta}}_{h} - \max\{1, k(t)\}C_{\tilde{\Lambda}} \langle \xi \rangle^{2-(2\sigma-1)}_{h} \langle x \rangle^{-\sigma}  \nonumber \\
	&\geq \left(M_2 \frac{C_{a_3}}{2} - C_{a_2} - \max\{1,k(t)\}C_{\lambda_2} h^{-(1-\frac{1}{\theta})}- \max\{1, k(t)\}C_{\tilde{\Lambda}}h^{-(2\sigma-1)}\right)\langle \xi \rangle^{2}_{h} \langle x \rangle^{-\sigma}, \nonumber
\end{align} 
For the terms of order $1$ we obtain
\begin{align} \label{lbestA_1}
	Re\, (\tilde{a}_1+c+e) &\geq 	M_1 \frac{C_{a_3}}{2} \langle \xi \rangle_{h} \langle x \rangle^{-\frac{\sigma}{2}} - C_{a_1} \langle \xi \rangle_{h} \langle x \rangle^{-\frac{\sigma}{2}} - C_{a_2, \lambda_2} \langle \xi \rangle_{h} \langle x \rangle^{-2\sigma} \\ 
	&- C_c \langle \xi \rangle_{h} \langle x \rangle^{-\sigma} - C_{e,k,\tilde{\Lambda}}\langle \xi \rangle^{\frac{1}{\theta}}_{h} \langle x\rangle^{-\sigma} \nonumber \\
	&\geq \left(M_1 \frac{C_{a_3}}{2} - C_{a_1} - C_{a_2, \lambda_2} - C_{c} - C_{e,k,\tilde{\Lambda}}h^{-(1-\frac{1}{\theta})}\right) \langle \xi \rangle_{h} \langle x \rangle^{-\frac{\sigma}{2}}.  \nonumber
\end{align}
Finally, for the terms of order $\frac{1}{\theta}$ we have
\begin{align} \label{lbestA_theta}
	Re\, \tilde{a}_{\theta} &\geq - k'(t)\langle \xi \rangle^{\frac{1}{\theta}}_{h}  - k(t) C_{\tilde{\Lambda}} \langle \xi \rangle^{\frac{1}{\theta}}_{h} 
	- C_{\tilde{\Lambda}} \langle \xi \rangle^{2(1-\sigma)}_{h} - 2^{\sigma}C_{a_3} M_2 \langle \xi \rangle^{2(1-\sigma)}_{h}  - 2^{\sigma}C_{a_3}M_1 \langle \xi \rangle^{1-\sigma}_{h} \nonumber \\
	&\geq - (k'(t) + C_1k(t) + C_2 ) \langle \xi \rangle^{\frac{1}{\theta}}_{h}, \nonumber
\end{align}
where $C_2 = \tilde{C}_{2}h^{-[\frac{1}{\theta}-2(1-\sigma)]}$ and $C_1, \tilde{C}_2$ depend on $\tilde{\Lambda}$ but not on  $h$. Setting 
$$
k(t) = e^{-C_1 t}k(0) - \frac{1-e^{-C_1 t}}{C_1} C_2, \quad t \in [0,T],
$$
we obtain $k'(t) \leq 0$ and $k'(t) + C_1k(t) + C_2 = 0$. Note that for any choice of $k(0) > 0$, we can choose $h$ large enough in order to obtain $k(t) > 0$ on $[0,T]$.

From the previous estimates from below we obtain the following proposition.   

\begin{proposition}\label{proposition_Hm_well_posedness_for_the_conjugated_problem}
	Let $k(0)$ be small enough such that \eqref{z} holds true. Then
	there exist $M_2 , M_1>0$ and a large parameter $h_{0}=h_0(k(0),M_2,M_1,T, \theta,\sigma)$ $> 0$ such that for every $h \geq h_0$ the Cauchy problem \eqref{equation_conjugated_cauchy_problem} is well-posed in Sobolev spaces $H^{m}(\R)$. More precisely, for any $\tilde f \in C([0,T];H^{m}(\R))$ and $\tilde g \in H^{m}(\R)$, there exists a unique solution $v \in C([0,T];H^{m}(\R)) \cap C^{1}([0,T]; H^{m-3}(\R))$ such that the following energy estimate holds
	$$
	\|v(t)\|^{2}_{H^{m}} \leq C \left( \| \tilde g \|^{2}_{H^{m}}  + \int_{0}^{t} \| \tilde f(\tau)\|^{2}_{H^{m}} d\tau \right), \quad t \in [0,T].
	$$
\end{proposition}
\begin{proof}
	Let $k(0) > 0$. Take $M_2 > 0$ such that 
	\begin{equation}\label{firstlbestimate}
		M_2 \frac{C_{a_3}}{2} - C_{a_2} > 0,
	\end{equation}
	and after that set $M_1 > 0$ in such way that
	\begin{equation} \label{secondlbestimate}
		M_1 \frac{C_{a_3}}{2} - C_{a_1} - C_{a_2, \lambda_2} - C_{c} > 0.
	\end{equation}
	Finally, making the parameter $h_{0}$ large enough, we obtain $k(T) > 0$ and
	\begin{equation} \label{thirdlbestimate}
		M_2 \frac{C_{a_3}}{2} - C_{a_2} - \max\{1,k(t)\}C_{\lambda_2} h^{-(1-\frac{1}{\theta})} - \max\{1, k(t)\}C_{\tilde{\Lambda}}h^{-(2\sigma-1)} \geq 0,
	\end{equation}
	\begin{equation}\label{fourthlbestimate}
		M_1 \frac{C_{a_3}}{2} - C_{a_1} - C_{a_2, \lambda_2} - C_{c} - C_{e,k,\tilde{\Lambda}}h^{-(1-\frac{1}{\theta})} > 0.
	\end{equation}
	With these choices $Re\, \tilde{a}_2(t,x,\xi), Re\, (\tilde{a}_1+c+e)(t,x,\xi), Re\, \tilde{a}_{\theta}(t,x,\xi)$ are non negative for large $|\xi|$. Applying the Fefferman-Phong inequality to $\Re \tilde a_2$ we have 
	$$
	Re \langle Re\, \tilde{a}_2(t,x,D) v,v \rangle_{L^{2}} \geq -C \|v \|^2_{L^2}, \quad v \in \mathscr{S}(\R).
	$$ 	
	By the sharp G{\aa}rding inequality we also obtain that
	$$
	Re \langle (\tilde{a}_1+c+e)(t,x,D) v,v \rangle_{L^{2}} \geq -C \|v \|^2_{L^2}, \quad v \in \mathscr{S}(\R)
	$$ 
	and 
	$$
	Re \langle \tilde{a}_\theta(t,x,D) v,v \rangle_{L^{2}} \geq -C \|v \|^2_{L^2}, \quad v \in \mathscr{S}(\R).
	$$ 
	As a consequence we get the energy estimate 
	$$
	\dfrac{d}{dt}\| v(t) \|^{2}_{L^{2}} \leq C' (\| v(t) \|^{2}_{L^{2}} + \| (iP)_{\Lambda} v(t) \|^{2}_{L^{2}}),
	$$ 
	which gives the well-posedness on $L^2(\R)$ and on $H^{m}(\R)$ for every $m \in \R$ for the Cauchy problem \eqref{equation_conjugated_cauchy_problem}.
\end{proof}

\section{Gevrey well-posedness for the Cauchy problem \eqref{Cauchy_problem_in_introduction}}\label{Proofmainresult}

Finally we are ready to prove Theorem $\ref{mainthm}$. 
\begin{proof}[Proof of Theorem \ref{mainthm}] Let us take initial data satisfying $f \in C([0,T], H^{m}_{\rho; \theta}(\R)), g \in H^{m}_{\rho; \theta}(\R)$, for some $m \in \R$ and $\rho > 0$. Now choose $k(0) < \rho$ and $M_2, M_1$ large enough so that Proposition \ref{proposition_Hm_well_posedness_for_the_conjugated_problem} holds true. We have 
$$
e^{\Lambda}(t,x,D)f \in C([0,T]; H^{m}_{\rho-(k(0) + \delta); \theta}(\R)), \quad e^{\Lambda}(0,x,D) g \in H^{m}_{\rho-(k(0) + \delta); \theta}(\R)
$$ 
for every $\delta > 0$, thanks to the continuity properties stated in Proposition \ref{contgev}. Since $k(0) < \rho$ and $k(t)$ is non-increasing, we may conclude $e^{\Lambda}(t,x,D)f$ is in $C([0,T]; H^{m}(\R))$ and $e^{\Lambda}g \in H^{m}(\R)$. Proposition \ref{proposition_Hm_well_posedness_for_the_conjugated_problem} gives $h_0 > 0$ large such that for $h \geq h_0$ the Cauchy problem associated with $P_{\Lambda}$ is well-posed in Sobolev spaces. Namely, there exists a unique $v \in C([0,T]; H^{m})$ satisfying 
$$
\begin{cases}
	P_{\Lambda} v(t,x) = e^{\Lambda}(t,x,D) f(t,x), \quad (t, x) \in [0,T] \times \R, \\
	v(0,x) = e^{\Lambda}(0,x,D)g(x), \qquad \qquad \,\,\, x \in \R^{n},
\end{cases}
$$ 
and
\begin{equation}\label{donkey_kong_1}
	\|v(t)\|^{2}_{H^{m}} \leq C \left( \| e^{\Lambda}g \|^{2}_{H^{m}}  + \int_{0}^{t} \|e^{\Lambda}f(\tau)\|^{2}_{H^{m}} d\tau \right), \quad t \in [0,T].
\end{equation}

Setting $u = \{e^{\Lambda}\}^{-1} v$ we obtain a solution for our original problem, that is 
$$
\begin{cases}
	P u(t,x) = f(t,x), \quad (t,x) \in [0,T] \times \R, \\
	u(0,x) = g(x), \qquad \qquad  x \in \R.
\end{cases}
$$
Now let us study which space the solution $u$ belongs to. We have 
$$
u = \{e^{\Lambda}\}^{-1} v = \,^{R}\{e^{-\tilde{\Lambda}}\} \underbrace{\sum_{j} (-r)^{j}}_{\text{order zero}} e^{-k(t)\langle D \rangle^{\frac{1}{\theta}}_{h}} v,
$$
where $v \in H^{m}$. Noticing that $k(T) > 0$ and $k$ is non-increasing, we achieve 
$$
e^{-k(t)\langle D \rangle^{\frac{1}{\theta}}_{h}} v = e^{-k(T)\langle D \rangle^{\frac{1}{\theta}}_{h}} 
\underbrace{e^{(k(T)-k(t))\langle D \rangle^{\frac{1}{\theta}}_{h}}}_{\text{order zero}} v \in H^{m}_{k(T); \theta} (\R).
$$
Hence $\{e^{\Lambda}(t,x,D)\}^{-1} v \in H^{m}_{k(T) - \delta; \theta}$ for every $\delta > 0$. Moreover, from \eqref{donkey_kong_1} we obtain that $u$ satisfies the following energy estimate 
\begin{align*}
	\|u(t)\|^{2}_{H^{m}_{k(T)-\delta;\theta}} &= \| \{e^{\Lambda}(t)\}^{-1} v(t)  \|^{2}_{H^{m}_{k(T)-\delta;\theta}}  \leq C_1 \| v(t) \|^{2}_{H^{m}} \\
	&\leq C_2 \left( \| e^{\Lambda}(0) g \|^{2}_{H^{m}}  + \int_{0}^{t} \|e^{\Lambda}(\tau) f(\tau)\|^{2}_{H^{m}} d\tau \right) \\
	&\leq C_3 \left( \| g \|^{2}_{H^{m}_{\rho;\theta}}  + \int_{0}^{t} \|f(\tau)\|^{2}_{H^{m}_{\rho;\theta}} d\tau \right) , \quad t \in [0,T].
\end{align*}

Summing up, given $f \in C([0,T], H^{m}_{\rho; \theta}(\R)), g \in H^{m}_{\rho; \theta}(\R)$ for some $m \in \R$ and $\rho > 0$, we find a solution $u \in C([0,T]; H^{m}_{\rho'; \theta}(\R))$ ($\rho' < \rho$) for the Cauchy problem associated with the operator $P$ and initial data $f,g$. 

Now it only remains to prove the uniqueness of the solution. To this aim, assume take $u_1, u_2 \in C([0,T];H^{m}_{\rho';\theta}(\R))$ such that 
$$
\begin{cases}
	Pu_j = f \\
	u_j(0)= g.
\end{cases}
$$ 
For a new choice of $k(0) < \rho'$ and applying once more Proposition $\ref{proposition_Hm_well_posedness_for_the_conjugated_problem}$, we may find new parameters $M_2, M_1 > 0$ and $h_0 > 0$ such that the Cauchy problem associated to
$$
P_{\Lambda} = e^{\Lambda} \circ P \circ \{e^{\Lambda}\}^{-1}
$$ 
is well-posed in $H^{m}$, where $e^\Lambda$ represents the operator corresponding to the transformation associated with these new parameters $k(0), M_2, M_1, h_0$. Since $e^{\Lambda}f, e^{\Lambda}g \in H^{m}$ and $u_j$, $j=1,2$, satisfy 
$$
\begin{cases}
	P_{\Lambda} e^{\Lambda} u_j = e^{\Lambda} f \\
	e^{\Lambda} u_j(0) = e^{\Lambda}g,
\end{cases}
$$
we must have $e^{\Lambda} u_1 = e^{\Lambda} u_2$ and therefore $u_1 = u_2$. This concludes the proof. 
\end{proof}

\begin{remark}
	In this paper we present a result in the one space dimensional case. The extension to higher space dimension requires a more involved choice of the functions $\lambda_1, \lambda_2$ which must satisfy certain partial differential inequalities, see for instance \cite{CRJEECT, ascanelli_cappiello_schrodinger_equations_Gelfand_shillov,  KB} in the case $p=2$ and the ideas in \cite[Section 4]{ABtame} for the case $p\geq 3.$ We prefer to not treat this extension for the moment because our aim in the next future is to apply this result to semilinear equations of physical interest defined for $x\in\R^1$.
\end{remark}

\begin{remark}\label{xnelleadingterm}
In Theorem \ref{mainthm} we assume that the symbol of the leading term $a_3(t,D)$ is independent of $x$. In the $H^\infty$ setting, it is possible to consider also the more general case $a_3(t,x,D)$, assuming for $a_3$ suitable decay estimates, see \cite[Section 4]{ABtame}.  This is not possible in the Gevrey setting using our arguments, due to the conjugation with $e^{k(t) \langle D \rangle^{1/\theta}}$; indeed, if $a_3$ depends on $x$, even allowing its derivatives with respect to $x$ to decay like $\langle x \rangle^{-m}$ for $m>>0$, we obtain
$$e^{k(t) \langle D \rangle^{1/\theta}} (ia_3(t,x,D)) e^{-k(t) \langle D \rangle^{1/\theta}} = ia_3(t,x,D_x) + \textrm{op }\left(k(t) \partial_\xi \langle \xi \rangle^{\frac{1}{\theta}}\cdot \partial_x a_3 \right) + \textrm{l.o.t}$$
with $k(t) \partial_\xi \langle \xi \rangle^{\frac{1}{\theta}}\cdot \partial_x a_3 (t,x,\xi) \sim \langle \xi \rangle^{2+\frac1{\theta}} \langle x \rangle^{-m}.$ This term has order $2+\frac{1}{\theta} >2$ and cannot be controlled by other lower order terms whose order does not exceed $2$.
\end{remark}

\noindent
\textbf{Acknowledgements.} The authors are grateful to professor Giovanni Taglialatela for helpful comments and suggestions. They also wish to express their gratitude to the referee for his/her valuable criticism which helped us to improve the presentation of the results.

\section{Appendix}

In this last Section we explain how to obtain Theorem \ref{conjthmnew} following the steps given in the proof of Theorems 6.9, 6.10 and 6.12 in \cite{KN}. We begin stating the following result, whose proof can be found in the final example of \cite[Section 6, Chapter 1]{Kumano-Go}.

\begin{lemma}\label{lemmaexample}
	Let $a \in \mathcal{B}^{\infty}(\R^{n})$. Then for every fixed $x \in \R^{n}$, the function $(y,\eta) \mapsto e^{ix\eta} a(y)$ belongs to a class of polynomially bounded amplitudes. Moreover
	$$
	Os- \iint e^{-iy\eta} e^{ix\eta} a(y) dy\dslash\eta = Os- \iint e^{-i\eta y} a(y+x) dy\dslash\eta = a(x), \quad x \in \R^{n}. 
	$$	
\end{lemma}

Theorem \ref{conjthmnew} is a direct consequence of the two following propositions and of the final Remark \ref{ultimoremark}.

\begin{proposition}\label{proposition_e_Lambda_circ_p}
	Under the assumptions of Theorem \ref{conjthmnew},
	there exists $\tilde{\delta} > 0$ such that if $\rho_0 \leq \tilde{\delta} A^{-\frac{1}{\kappa}}$, then 
	$$
	e^\lambda(x,D) \circ p(x,D) = \textrm{op}\left(e^\lambda(x,\xi) s_N(x,\xi)\right) + q_N(x,D)+ r_{\infty}(x,D),  
	$$
	where 
	\begin{equation}\label{sN}
		s_N(x,\xi) = \sum_{|\alpha|<N } \frac{1}{\alpha!} e^{-\lambda(x,\xi)} \{\partial^{\alpha}_{\xi}e^{\lambda(x,\xi)}\} D^{\alpha}_{x}p(x,\xi),
	\end{equation}
	\begin{equation}\label{eq_remainder1}
		|\partial^{\alpha}_{\xi}\partial^{\beta}_{x}q_N(x,\xi)| \leq C_{\rho_0,A,\kappa} (C_{\kappa}A)^{|\alpha+\beta|+2N}\alpha!^{\kappa}\beta!^{\kappa}N!^{2\kappa-1} \langle \xi \rangle^{m-(1-\frac{1}{\kappa})N - |\alpha|},
	\end{equation}
	\begin{equation}\label{eq_remainder2}
		|\partial^{\alpha}_{\xi}\partial^{\beta}_{x}r_{\infty}(x,\xi)| \leq C_{\rho_0,A,\kappa} (C_{\kappa}A)^{|\alpha+\beta|+2N}\alpha!^{\kappa}\beta!^{\kappa}N!^{2\kappa-1} e^{-c_\kappa A^{-\frac{1}{\kappa}} \langle \xi \rangle^{\frac{1}{\kappa}}}.
	\end{equation}                    
\end{proposition}
\begin{proof}
	Arguing as in the proof of \cite[Theorem 6.9]{KN}, we can write the symbol $s(x,\xi)$ of the composition $e^\lambda(x,D) \circ p(x,D)$ as
	\begin{align*}
		s(x,\xi) = Os- \iint e^{-i\eta y} e^{\lambda(x, \xi + \eta)} p(x+y,\xi) dy\dslash\eta.
	\end{align*}
	Applying Taylor's formula to $e^{\lambda(x,\xi+\eta)}$ and then applying Lemma \ref{lemmaexample} we obtain 
	\begin{align*}
		s(x,\xi) =  \sum_{|\alpha| < N} \frac{1}{\alpha!} \partial^{\alpha}_{\xi}e^{\lambda(x,\xi)} D^{\alpha}_{x}p(x,\xi)
		+ r_{N}(x,\xi)= e^{\lambda(x,\xi)} s_N(x,\xi)+r_N(x,\xi),
	\end{align*}	
	where 
	\begin{align*}
		r_N(x,\xi) = \sum_{|\alpha| = N} \frac{1}{\alpha!} \,Os-\iint e^{-iy\eta}D^{\alpha}_{x} p(x+y,\xi) \int_{0}^{1} (1-\theta)^{N-1} \partial^{\alpha}_{\xi} e^{\lambda(x,\xi+\theta\eta)} d\theta  dy\dslash\eta.
	\end{align*}
	Therefore 
	\begin{align*}
		e^\lambda(x,D) \circ p(x,D) =  \textrm{op}\left(e^\lambda(x,\xi)  s_N(x,\xi)\right) +  r_N(x,D),
	\end{align*}
where $s_N$ is given by \eqref{sN}. 
	Take now $\chi(t) \in C^{\infty}_{c}(\R)$ such that 
	\beqs\label{chi}
	|\partial^{j}_{t}\chi(t)| \leq C^{j+1} j!^{\kappa'}\,\, (1< \kappa' < \kappa), \quad \chi(t) =
	\begin{cases}
		1, \quad |t| \leq \frac{1}{4} \\
		0, \quad |t| \geq \frac{1}{2}
	\end{cases},
	\eeqs
and	set $\chi(\xi, \eta) = \chi(\langle \eta \rangle \langle \xi \rangle^{-1})$, $\xi, \eta \in \R^{n}$. Note that 
	$$
	\frac{1}{2} \langle \xi \rangle \leq \langle \xi + \theta\eta \rangle \leq \frac{3}{2} \langle \xi \rangle
	$$
	for every $\xi, \eta \in \, \text{supp}\, \chi(\xi, \eta)$ and $|\theta| \leq 1$. We can split the operator $r_N(x,D)$ as 
	\begin{align*}
 r_N(x,D)u(x) &= \sum_{|\alpha| = N} \frac{1}{\alpha!} \int e^{i\xi x+\lambda(x,\xi)}\, Os- \iint e^{-i\eta y} \int_{0}^{1} (1-\theta)^{N-1} e^{\lambda(x,\xi+\theta\eta)-\lambda(x,\xi)} \\
		&\times w^{\alpha}(\lambda; x, \xi+\theta\eta) d\theta\, \{D^{\alpha}_{x}p\}(x+y,\xi) dy\dslash\eta \, \widehat{u}(\xi) \dslash\xi	 \\
		&= \int e^{i\xi x+\lambda(x,\xi)} r'_N(x,\xi) \widehat{u}(\xi) \dslash\xi + \int e^{ix\xi} r^{''}_N(x,\xi) \widehat{u}(\xi) \dslash\xi,
	\end{align*}
	where, for $\alpha,\beta\in \N^n$ we denote $w^\alpha(\lambda;x,\xi):=e^{-\lambda(x,\xi)}\partial_\xi^\alpha  e^\lambda(x,\xi)$ and
	\begin{align*}
		r'_N(x,\xi) &= \lim_{\varepsilon  \to 0} \sum_{|\alpha| = N} \frac{1}{\alpha!}\iint e^{-i\eta y} \int_{0}^{1} (1-\theta)^{N-1} e^{\lambda(x,\xi+\theta\eta)-\lambda(x,\xi)}w^{\alpha}(\lambda; x, \xi+\theta\eta) d\theta \\
		&\times  \{D^{\alpha}_{x}p\}(x+y,\xi) \chi(\xi, \eta) \chi_{\varepsilon}(y,\eta) dy\dslash\eta,
	\end{align*}
	\begin{align*}
		r^{''}_N(x,\xi) &= \lim_{\varepsilon  \to 0} \sum_{|\alpha| = N} \frac{1}{\alpha!}\iint e^{-i\eta y} \int_{0}^{1} (1-\theta)^{N-1} 
		\partial^{\alpha}_{\xi} e^{\lambda(x,\xi+\theta\eta)} d\theta \\
		&\times  \{D^{\alpha}_{x}p\}(x+y,\xi) (1-\chi)(\xi, \eta) \chi_{\varepsilon}(y,\eta) dy\dslash\eta,
	\end{align*}
	and $\chi_{\varepsilon}(y,\eta) = \chi(\varepsilon y)\chi(\varepsilon\eta)$, $\chi \in \Sigma_{\kappa'}(\R^{n})$, $\chi(0) = 1$.
	It is not difficult to verify that the term $D^{\alpha}_{x}p(x+y,\xi)e^{-\lambda(x, \xi+\theta\eta)}\partial^{\alpha}_{\xi}e^{\lambda(x,\xi+\theta\eta)}$ has order $m-N(1-\frac{1}{\kappa})$ with respect to $\xi$. Applying the same arguments used in \cite[Theorem 6.9]{KN} to estimate the remainder terms of the composition, one can split the symbol $r'_N$ into the sum of two terms satisfying \eqref{eq_remainder1} and \eqref{eq_remainder2} respectively. The estimate \eqref{eq_remainder2} for $r^{''}_N$ can be obtained simply integrating by parts and using the fact that $(1-\chi)(\xi,\eta)$ is supported for $\langle \eta \rangle \geq 4^{-1} \langle \xi \rangle$. We leave the details to the reader. 
\end{proof}

\begin{proposition}\label{proposition_p_phi_circ_reverse} Under the assumptions of Theorem \ref{conjthmnew},
	there exist $\tilde{\delta} > 0$ and $h_0(A) \geq 1$ such that if $h \geq h_0$ and $\rho_0 \leq \tilde{\delta} A^{-\frac{1}{\kappa}}$ we may write the product $\textrm{op}(e^\lambda p)\circ\, ^{R}\{e^{-\lambda}\}$ as follows
	$$
	\textrm{op}(e^\lambda p)\circ\, ^{R}\{e^{-\lambda}(x,D)\} = s_{N'}(x,D) + q_{N'}(x,D) + r_{\infty}(x,D),  
	$$
	where 
	\begin{equation}
		s_{N'}(x,\xi) = \sum_{|\alpha| < N'} \frac{1}{\alpha!} \partial^{\alpha}_{\xi} \{ e^{\lambda(x,\xi)}p(x,\xi) D^{\alpha}_{x}e^{-\lambda(x,\xi)} \},
	\end{equation}
	\begin{equation}\label{equation_estimate_q_N_composition_with_reverse}
		|\partial^{\alpha}_{\xi}\partial^{\beta}_{x}q_{N'}(x,\xi)| \leq C_{\rho_0,A,\kappa} (C_{\kappa}A)^{|\alpha+\beta|+2N'}\alpha!^{\kappa}\beta!^{\kappa}N'!^{2\kappa-1} \langle \xi \rangle^{m-(1-\frac{1}{\kappa})N' - |\alpha|}_{h},
	\end{equation}
	\begin{equation}\label{equation_estimate_regularizing_term_with_reverse}
		|\partial^{\alpha}_{\xi}\partial^{\beta}_{x}r_{\infty}(x,\xi)| \leq C_{\rho_0,A,\kappa} (C_{\kappa}A)^{|\alpha+\beta|+2N'}\alpha!^{\kappa}\beta!^{\kappa}N'!^{2\kappa-1} e^{-c_\kappa A^{-\frac{1}{\kappa}} \langle \xi \rangle^{\frac{1}{\kappa}}_{h}}.
	\end{equation}
\end{proposition}

\begin{proof}
	We have 
	\begin{align*}
		^{R}\{e^{-\lambda}(x,D)\} u(x) &= \iint e^{i\xi (x-y)} e^{-\lambda(y,\xi)} u(y) dy \dslash\xi \\
		&= \int e^{i\xi x} \int e^{-i\xi y} e^{-\lambda(y,\xi)} u(y) dy \,\dslash\xi,
	\end{align*}
	which implies
	$$		
		\textrm{op}(e^\lambda p)\circ\, ^{R}\{e^{-\lambda}(x,D)\} u(x) 
= \iint e^{i\xi (x-y)} e^{\lambda(x,\xi) - \lambda(y,\xi)} p(x,\xi) u(y) dy \dslash\xi.
	$$
	In this way the symbol $\sigma(x,\xi)$ of the composition $\textrm{op}(e^\lambda p)\circ\, ^{R}\{e^{-\lambda}\}$ is given by
	\begin{align*}
		\sigma(x,\xi) = Os- \iint e^{-i\eta y} 	e^{\lambda(x,\xi+\eta) - \lambda(x+y,\xi+\eta)} p(x,\xi+\eta) dy\dslash\eta.
	\end{align*}
	By Taylor's formula and Lemma \ref{lemmaexample} we obtain 
	\begin{align*}
		\sigma(x,\xi) &= \sum_{|\alpha'| < N'} \frac{1}{\alpha!} \partial^{\alpha}_{\xi} \{ e^{\lambda(x,\xi)}p(x,\xi) D^{\alpha}_{x}e^{-\lambda(x,\xi)} \} \\
		&+ N'\sum_{|\alpha| = N'} \frac{1}{\alpha!} \int_{0}^{1} (1-\theta)^{N'-1} \, Os-\iint e^{-iy\eta} \partial^{\alpha}_{\xi} \{
		e^{\lambda(x,\xi+\theta\eta)}p(x,\xi+\theta\eta) D^{\alpha}_{x}e^{-\lambda(x+y,\xi+\theta\eta)} \}  dy\dslash\eta\,d\theta \\
	&= s_{N'}(x,\xi)+ r_{N'}(x,\xi).
	\end{align*}
	Now we observe that thanks to \eqref{equation_stronger_hypothesis_on_Lambda} we have that 
	$$
	e^{-\lambda(x,\xi+\theta\eta)}e^{\lambda(x+y,\xi+\theta\eta)}
	\partial^{\alpha}_{\xi} \{e^{\lambda(x,\xi+\theta\eta)}p(x,\xi+\theta\eta) D^{\alpha}_{x}e^{-\lambda(x+y,\xi+\theta\eta)} \}
	$$
	has order $m-N'(1-\frac{1}{\kappa})$ (with respect to $\xi$). Applying the same argument used
	   in the proof of \cite[Theorem 6.10]{KN}, one can split the symbol $r_{N'} = q_{N'} + r_{\infty}$, where $q_{N'}$ and $r_{\infty}$ satisfy \eqref{equation_estimate_q_N_composition_with_reverse} and \eqref{equation_estimate_regularizing_term_with_reverse} respectively. Details are left to the reader.

\end{proof}

\begin{remark}\label{ultimoremark}
	Shrinking $\tilde{\delta} > 0$ if necessary, we may conclude that 
	$$
	r_{\infty}(x,D) \circ\, ^{R}\{e^{-\lambda}(x,D)\} = \tilde{r}_{\infty}(x,D),
	$$
	where $\tilde{r}_\infty(x,D)$ is still a regularizing operator.
\end{remark}





\end{document}